\documentclass[a4paper,
	11pt,
	pdftex,
	headings=normal,
	headsepline,
	footsepline,
	onecolumn,
	headinclude,
	footinclude,
	DIV14,
	abstracton
]{scrartcl}

\usepackage{
	graphicx,
	amssymb,
	amsmath,
	amsthm,
	xcolor,
	dsfont,
	algpseudocode,
	authblk,
}

\usepackage[ngerman, english]{babel}

\usepackage[bf]{caption}
\usepackage{amssymb}
\usepackage{algpseudocode}

\usepackage{enumerate}
\usepackage{algorithm}
\usepackage{subcaption}    
\usepackage{tabularx}
\usepackage[left=2.50cm, right=2.50cm, top=2.00cm, bottom=3.50cm]{geometry}

\newtheorem{theorem}{Theorem}[section]

\newtheorem{lemma}[theorem]{Lemma}

\newtheorem{definition}[theorem]{Definition}
\newtheorem{remark}[theorem]{Remark}
\newtheorem{example}[theorem]{Example}

\let\OLDthebibliography\thebibliography
\renewcommand\thebibliography[1]{
	\OLDthebibliography{#1}
	\setlength{\parskip}{0pt}
	\setlength{\itemsep}{0pt plus 0.3ex}
}

\begin{document}
	\title{ROM-based multiobjective optimization of elliptic PDEs via numerical continuation}
	\author[1]{Stefan Banholzer}
	\author[2]{Bennet Gebken}
	\author[2]{Michael Dellnitz}
	\author[2]{Sebastian Peitz}
	\author[1]{Stefan Volkwein}
	\affil[1]{\normalsize Department of Mathematics and Statistics, University of Konstanz, Germany}
	\affil[2]{\normalsize Department of Mathematics, Paderborn University, Germany}

	\maketitle

	\begin{abstract}
	Multiobjective optimization plays an increasingly important role in modern applications, where several objectives are often of equal importance. The task in multiobjective optimization and multiobjective optimal control is therefore to compute the {set of optimal compromises} (the \emph{Pareto set}) between the conflicting objectives. 
	Since the Pareto set generally consists of an infinite number of solutions, the computational effort can quickly become challenging which is particularly problematic when the objectives are costly to evaluate as is the case for models governed by partial differential equations (PDEs). To decrease the numerical effort to an affordable amount, surrogate models can be used to replace the expensive PDE evaluations.
	Existing multiobjective optimization methods using model reduction are limited either to low parameter dimensions or to few (ideally two) objectives. In this article, we present a combination of the reduced basis model reduction method with a continuation approach using inexact gradients. The resulting approach can handle an arbitrary number of objectives while yielding a significant reduction in computing time.
	\end{abstract}

	\section{Introduction}
	The dilemma of deciding between multiple, equally important goals is present in almost all areas of engineering and economy. A prominent example comes from production, where we want to produce a product at minimal cost while simultaneously preserving a high quality. In the same manner, multiple goals are present in most technical applications, maximizing the velocity while minimizing the energy consumption of electric vehicles \cite{PSOB17} being only one of many examples. These conflicting goals result in \emph{multiobjective optimization problems} (MOPs) \cite{Ehr05}, where we want to optimize all objectives simultaneously. Since the objectives are in general contradictory, there exists an infinite number of \emph{optimal compromises}. The set of these compromise solutions is called the \emph{Pareto set}, and the goal in multiobjective optimization is to approximate this set in an efficient manner, which is significantly more expensive than solving a single objective problem. Due to this, the development of efficient numerical approximation methods is an active area of research, and methods range from scalarization \cite{Ehr05,H2001} over set-oriented approaches \cite{DSH2005} and continuation \cite{H2001} to evolutionary algorithms \cite{CLV07}. Recent advances have paved the way to new challenging application areas for multiobjective optimization such as feedback control or problems constrained by partial differential equations (PDEs); cf.~\cite{PD18b} for a survey. 
	
	In the presence of PDE constraints, the computational effort can quickly become infeasible such that special means have to be taken in order to accelerate the computation.
	To this end, surrogate models form a promising approach for significantly reducing the computational effort. A widely used approach is to directly construct a mapping from the parameter to the objective space using as few function evaluations of the expensive model as possible, cf.~\cite{THH+15,CSHM17} for extensive reviews. In the case of PDE constraints, an alternative approach is via dimension reduction techniques such as Proper Orthogonal Decomposition (POD) \cite{Sir87,KV01} or the reduced basis (RB) method \cite{GP05}. In these methods, a small number of high-fidelity solutions is used to construct a low-dimensional surrogate model for the PDE which can be evaluated significantly faster while guaranteeing convergence using error estimates. In recent years, several methods have been proposed where model reduction is used in multiobjective optimization and optimal control. In \cite{IUV17} and \cite{ITV16}, scalarization using the so-called weighted sum method was combined with RB and POD, respectively. In \cite{BBV16,BBV17}, convex problems were solved using reference point scalarization and POD, and set-oriented approaches were used in \cite{BDPV18,BDPV17}. A comparison of both was performed in \cite{POBD19} for the Navier--Stokes equations.
	
	In this article we combine an extension of the continuation methods presented in \cite{H2001,SDD2005} to inexact gradients (Section \ref{p5sec:ContinuationMethodWithInexactness}) with a reduced basis approach for elliptic PDEs (Section~\ref{p5sec:MultiobjectiveParameterOptimizationWithRB}). To deal with the error introduced by the RB approach, we combine the KKT conditions for MOPs with error estimates for the RB method to obtain a tight superset of the Pareto set. For the example considered here, the proposed method yields a speed-up factor of approximately 63 compared to the direct solution of the expensive problem (Section~\ref{p5sec:NumericalResults}). Additionally, our approach allows us to control the quality of the result by controlling the errors for each objective function individually.

\section{A continuation method for MOPs with inexact objective gradients} \label{p5sec:ContinuationMethodWithInexactness}
	In this section, we will begin by briefly introducing the basic concepts of multiobjective optimization upon which we will build in this article (see \cite{Ehr05,H2001} for detailed introductions). Afterwards, we will discuss the continuation method for MOPs and present two modifications of it that can deal with inexact gradient information. 	
	
	\subsection{Multiobjective optimization} \label{p5sec:MultiobjectiveOptimization}
	The goal of multiobjective optimization is to minimize several conflicting criteria at the same time. In other words, we want to minimize an objective $J = (J_1,...,J_k) : \mathbb{R}^n \rightarrow \mathbb{R}^k$ that is vector valued. It maps the \emph{variable space} $\mathbb{R}^n$ to the \emph{image space} $\mathbb{R}^k$. In contrast to single-objective optimization (i.e., $k = 1$), there exists no natural total order of the image space $\mathbb{R}^k$ for $k > 1$. As a result, the classical concept of \emph{optimality} has to be generalized:
	
	\begin{definition} \label{p5def:Pareto_optimal}
		\begin{itemize}
			\item[(a)] $\bar{u} \in \mathbb{R}^n$ is called \emph{(globally) Pareto optimal} if there is no other point $u \in \mathbb{R}^n$ such that $J_i(u) \leq J_i(\bar{u})$ for all $i \in \{1,...,k\}$ and $J_j(u) < J_j(\bar{u})$ for some $j \in \{1,...,k\}$. 
			\item[(b)] The set $P$ of all Pareto optimal points is called the \emph{Pareto set}. Its image under $J$ is the \emph{Pareto front}.
		\end{itemize}
	\end{definition}
	
	The Pareto set is the solution of the \emph{multiobjective optimization problem (MOP)}
	\begin{equation} \label{p5eq:MOP}
		\min_{u \in \mathbb{R}^n} J(u). \tag{MOP}
	\end{equation}
	Constrained MOPs can be formulated analogously by restricting $u$ in Definition \ref{p5def:Pareto_optimal} to a subset $U \subseteq \mathbb{R}^n$. Similar to the scalar-valued case, if $J$ is differentiable, we can use the derivative of $J$ to obtain necessary conditions for Pareto optimality, the \emph{Karush-Kuhn-Tucker (KKT) conditions} \cite{H2001}:
	\begin{theorem}
		Let $\bar{u}$ be a Pareto optimal point of \eqref{p5eq:MOP}. Then there exist multipliers
		\begin{equation*}
			\alpha \in \Delta_k := \left\{ \alpha \in (\mathbb{R}^{\geq 0})^k : \sum_{i = 1}^k \alpha_i = 1 \right\}
		\end{equation*}
		such that
		\begin{equation} \label{p5eq:KKT} \tag{KKT}
			DJ(\bar{u})^\top \alpha = \sum_{i = 1}^k \alpha_i \nabla J_i(\bar{u}) = 0.
		\end{equation}
	\end{theorem}

	For $k = 1$, this reduces to the well-known optimality condition $\nabla J(\bar{u}) = 0$. If $J$ is non-convex, then the points satisfying \eqref{p5eq:KKT} form a proper superset of the Pareto set $P$:
	\begin{definition}
		If $\bar{u} \in \mathbb{R}^n$ and $\bar{\alpha} \in \Delta_k$ satisfy \eqref{p5eq:KKT}, then $\bar{u}$ is called \emph{Pareto critical} with corresponding \emph{KKT vector} $\bar{\alpha}$, containing the \emph{KKT multipliers} $\bar{\alpha}_i$, $i \in \{1,...,k\}$. The set $P_c$ of all Pareto critical points is called the \emph{Pareto critical set}.
	\end{definition}
	When solving an MOP, an initial step can be to compute the Pareto critical set. This set possesses additional structure which can be exploited in numerical schemes. Introducing the function
	\begin{equation*}
		F : \mathbb{R}^n \times (\mathbb{R}^{>0})^k \rightarrow \mathbb{R}^{n+1},
		(u,\alpha) \mapsto 
		\begin{pmatrix}
			\sum_{i = 1}^k \alpha_i \nabla J_i(u) \\
			1 - \sum_{i = 1}^k \alpha_i 		
		\end{pmatrix},
	\end{equation*}
	we see that Pareto critical points and their corresponding KKT vectors can be described as the zero level set of $F$. As shown by Hillermeier \cite{H2001}, this has the following implication:
	
	\begin{theorem} \label{p5thm:hillermeier}
		Let $J$ be twice continuously differentiable.
		\begin{itemize}
			\item[(a)] Let $\mathcal{M} := \{ (u,\alpha) \in \mathbb{R}^n \times (\mathbb{R}^{>0})^k : F(u,\alpha) = 0 \}$. If the Jacobian of $F$ has full rank everywhere, i.e.,
			\begin{equation} \label{p5eq:rank_DF}
				rk(DF(u,\alpha)) = n + 1 \quad \forall (u,\alpha) \in \mathcal{M},
			\end{equation}
			then $\mathcal{M}$ is a $(k-1)$-dimensional differentiable submanifold of $\mathbb{R}^{n+k}$. The tangent space of $\mathcal{M}$ at $(u,\alpha)$ is given by
			\begin{equation*}
				T_{(u,\alpha)} \mathcal{M} = ker(DF(u,\alpha)).
			\end{equation*}
			\item[(b)] Let $(u,\alpha) \in \mathcal{M}$ such that \eqref{p5eq:rank_DF} holds in $(u,\alpha)$. Then there is an open set $U \subseteq \mathbb{R}^n \times \mathbb{R}^k$ with $(u,\alpha) \in U$ such that $\mathcal{M} \cap U$ is a manifold as in \emph{(a)}. In other words, $\mathcal{M}$ locally possesses a manifold structure in all points satisfying \eqref{p5eq:rank_DF}.
		\end{itemize}
	\end{theorem}
		
	Theorem \ref{p5thm:hillermeier} forms the basis for the continuation method we use in this article.
	
	\subsection{Continuation method with exact gradients}
	\label{p5subsec:ExactContinuation}
	We only give a brief description of the method here and refer to \cite{SDD2005} and \cite{H2001} for details. By Theorem \ref{p5thm:hillermeier}, the Pareto critical set is -- except for the boundary -- the projection of the differentiable manifold $\mathcal{M} \subseteq \mathbb{R}^n \times \mathbb{R}^k$ onto its first $n$ components. In \cite{GPD2019} it has been shown that generically, this also holds for the first-order approximations, i.e., the projection of the tangent space of $\mathcal{M}$ yields the tangent cone of $P_c$. Given a Pareto critical point $\bar{u} \in P_c$, this means that we can find first-order candidates for new Pareto critical points in the vicinity of $\bar{u}$ by moving in the projected tangent space of $\mathcal{M}$. The idea of the continuation method is to do this iteratively to explore the entire Pareto critical set.
	
	Instead of approximating $P_c$ by a set of points, we use a set-oriented numerical approach; cf.~\cite{SDD2005} for details. This has the key advantage that it is easy to check whether a certain part of the set has already been computed, which is difficult when working with points. Additionally, a covering of $P_c$ by boxes makes it easy to obtain (and exploit) its topological properties. 
	In the approach, we evenly divide the variable space $\mathbb{R}^n$ into hypercubes or \emph{boxes} $B$ with radius $r > 0$:
	\begin{align} \label{p5eq:def_boxes}
		\mathcal{B}(r) &:= \{ [-r,r]^n + (2 i_1 r,...,2 i_n r)^\top : (i_1,...,i_n) \in \mathbb{Z}^n \}.
	\end{align}	
	
	\begin{remark}
		For ease of notation and readability, we will only consider the case where points $u \in \mathbb{R}^n$ are contained in single boxes. In other words, we only consider the case where $u$ is in the interior of a box and not in the intersection of multiple boxes. Since this is the generic case, this has no impact on the numerical methods we will propose later. \hfill$\blacksquare$
	\end{remark}	
	
	For $u \in \mathbb{R}^n$ let $B(u,r)$ be the box containing $u$. We want to compute the subset of $\mathcal{B}(r)$ covering the Pareto critical set for a given radius $r$, i.e., 
	\[
		\mathcal{B}_c(r) := \{ B \in \mathcal{B}(r) : P_c \cap B \neq \emptyset \}.
	\]
	Since we are interested in a covering via boxes instead of an approximation via points, when moving in a tangent direction of the critical set, we will search for \emph{tangent boxes} instead of single points. For $u \in \mathbb{R}^n$ let
	\begin{equation*}
		N(u,r) := \{ B \in \mathcal{B}(r) : B(u,r) \cap B \neq \emptyset \}
	\end{equation*}
	be the set of neighboring boxes of $B(u,r)$. Starting from a box $B(\bar{u},r)$ containing a critical point $\bar{u}$ with KKT vector $\bar{\alpha}$, we want to explore the neighboring boxes covering the projected tangent space at $\bar{u}$, i.e.,
	\begin{equation} \label{p5eq:exact_tangent_boxes}
		\mathcal{B}'(\bar{u},r) = \{ B' \in \mathcal{B}(r) : B' \in N(\bar{u},r), \ B' \cap \bar{u} + pr_u(T_{(\bar{u},\bar{\alpha})} \mathcal{M}) \neq \emptyset \}.
	\end{equation}	
	Here, $pr_u : \mathbb{R}^{n+k} \rightarrow \mathbb{R}^n$ is the projection of the tangent space onto the first $n$ components, i.e., the variable space. The typical situation is visualized in Figure \ref{p5fig:exact_cont}.
	
	\begin{figure}[ht]
		\centering
		\includegraphics[width=0.5\textwidth]{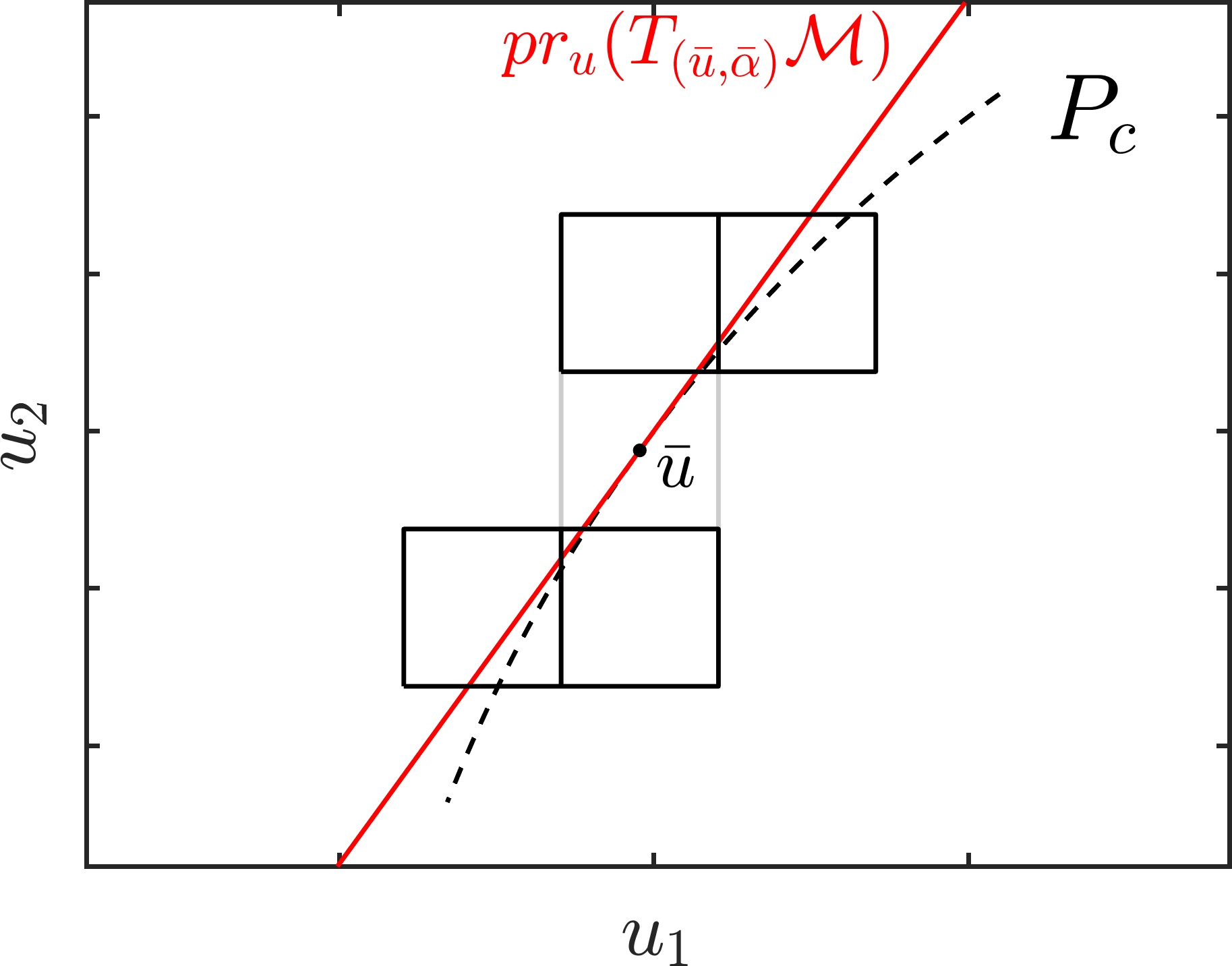}
		\caption{Tangent boxes (black) of the initial box (grey) containing $\bar{u}$, which is contained in the Pareto critical set $P_c$ (dashed). The red line indicates the projection of the tangent space of $\mathcal{M}$ onto the variable space}
		\label{p5fig:exact_cont}
	\end{figure}		
	
	As the tangent space of the Pareto critical set is only a linear approximation, a corrector step is required to verify that a given tangent box actually contains part of the Pareto critical set. This means that there has to be at least one $u\in B$ satisfying \eqref{p5eq:KKT}. To this end, for a box $B$, we consider the problem
	\begin{equation}\label{p5eq:KKT_box_test}\tag{PC-Box}
		\min_{u \in B, \alpha \in \Delta_k}  \| DJ(u)^\top \alpha \|_2^2
	\end{equation}
	Let $\theta(B)$ be the optimal value of this problem. Then obviously 
	\begin{equation*}
		\theta(B) = 0 \Leftrightarrow B \cap P_c \neq \emptyset.
	\end{equation*}
	In particular, if $\theta(B) = 0$ and $(\bar{u},\bar{\alpha})$ is the solution of \eqref{p5eq:KKT_box_test}, then $\bar{u}$ is Pareto critical with corresponding KKT vector $\bar{\alpha}$. After solving \eqref{p5eq:KKT_box_test} in each tangent box, all boxes with $\theta(B) = 0$ are added to a queue and a new iteration of the method is started with the first element in the queue. The method stops when the queue is empty, i.e., when there is no neighboring box of the current set of boxes that contains part of the Pareto critical set. For the remainder of this article, we will refer to this method as the \emph{exact continuation method}.
		
	\subsection{Continuation method with inexact gradients}	\label{p5subsec:ContinuationMethodInexact}	
	Using ROM to solve the state equation of an MOP of an elliptic PDE will introduce an error in the objective functions and the corresponding gradients, which has to be taken into account in order to ensure Pareto criticality of the solution. We here present a method that calculates a tight superset of the Pareto critical set via numerical continuation, using upper bounds for the errors in the approximated gradients. Formally, we now assume that for each gradient $\nabla J_i$, we only have an approximation $\nabla J_i^r$ such that
	\begin{equation} \label{p5eq:error_bounds}
		\sup_{u \in \mathbb{R}^n} \left\lVert \nabla J_i(u) - \nabla J^r_i(u) \right\rVert_2 \leq \epsilon_i, \ i \in \{1,...,k\},
	\end{equation}
	with upper bounds $\epsilon = (\epsilon_1, ..., \epsilon_k)^\top \in \mathbb{R}^k$.	Let $P_c$ and $P^r_c$ be the Pareto critical sets corresponding to $(\nabla J_i)_i$ and $(\nabla J_i^r)_i$, respectively. The following lemma shows how these error bounds translate to error bounds for the KKT conditions:
	
	\begin{lemma} \label{p5lem:KKT_error}
		Let $\bar{u} \in \mathbb{R}^n$ be Pareto critical for $J$ with KKT vector $\bar{\alpha} \in \Delta_k$. Then
		\begin{equation*}
			\| DJ^r(\bar{u})^\top \bar{\alpha} \|_2 \leq \sum_{i = 1}^k \bar{\alpha}_i \epsilon_i \leq \| \epsilon \|_\infty.
		\end{equation*}
	\end{lemma}
	\begin{proof}
		From the estimate
		\begin{align*}
			\left\lVert DJ^r(\bar{u})^\top \bar{\alpha} \right\rVert_2 
			&= \left\lVert DJ^r(\bar{u})^\top \bar{\alpha} - DJ(\bar{u})^\top \bar{\alpha} \right\rVert_2 
			= \left\lVert \sum_{i = 1}^k (\nabla J_i^r(\bar{u}) - \nabla J_i(\bar{u}))^\top \bar{\alpha}_i \right\rVert_2 \\
			&\leq \sum_{i = 1}^k  \left\lVert \nabla J_i^r(\bar{u}) - \nabla J_i(\bar{u}) \right\rVert_2 \bar{\alpha}_i 
			\leq \sum_{i = 1}^k \bar{\alpha}_i \epsilon_i 
			\leq \left\lVert \epsilon \right\rVert_\infty
		\end{align*}
		we derive the claim.
	\end{proof}			
		
	\begin{remark}
		Lemma \ref{p5lem:KKT_error} can be generalized to equality and inequality constrained MOPs using the constrained version of the optimality conditions from \cite{H2001}. In this case, in the norm on the left-hand side of the inequality in Lemma \ref{p5lem:KKT_error}, one additionally has to add a linear combination of the gradients of the equality and inequality constraints. \hfill$\blacksquare$
	\end{remark}			
		
	Lemma \ref{p5lem:KKT_error} shows that we have to weaken the conditions for Pareto criticality of the reduced objective function to obtain a superset of the actual Pareto critical set $P_c$. Formally, let 
	\begin{align*}
		P_1^r &:= \left\{ u \in \mathbb{R}^n : \min_{\alpha \in \Delta_k} \| DJ^r(u)^\top \alpha \|_2^2 \leq \| \epsilon \|_\infty^2 \right\}, \\
		P_2^r &:= \left\{ u \in \mathbb{R}^n : \min_{\alpha \in \Delta_k} \left( \| DJ^r(u)^\top \alpha \|_2^2 - (\alpha^\top \epsilon)^2 \right) \leq 0 \right\}.
	\end{align*}
	$P_1^r$ was also considered in \cite{PD2017} in the context of descent directions, where the solution of $\min_{\alpha \in \Delta_k} \| DJ^r(u)^\top \alpha \|_2^2$ is the squared length of the steepest descent direction in $u$. The condition for a point being in $P_1^r$ only depends on the maximal error $\| \epsilon \|_\infty$ and can be seen as a relaxed version of the KKT conditions for the inexact objective function. In contrast to this, the condition in $P_2^r$ actually considers the individual error bounds. By Lemma \ref{p5lem:KKT_error}, 
	\begin{equation*}
		P_c \subseteq P_2^r \subseteq P_1^r \text{ and } P^r_c \subseteq P_2^r \subseteq P_1^r,
	\end{equation*}
	i.e., both $P_1^r$ and $P_2^r$ are supersets of $P_c$ and $P_c^r$ (the points $\bar{u}$ for which the inexact gradients satisfy \eqref{p5eq:KKT}). In fact, $P_2^r$ is a tight superset of $P_c$ in the following sense:
	\begin{lemma} \label{p5lem:P2_tight}
		Let $\tilde{u} \in P_2^r$. Then there is some continuously differentiable $\tilde{J} : \mathbb{R}^n \rightarrow \mathbb{R}^k$ with 
		\begin{equation*}
			\sup_{u \in \mathbb{R}^n} \| \nabla \tilde{J}_i(u) - \nabla J^r_i(u) \|_2 \leq \epsilon_i \ \forall i \in \{1,...,k\}
		\end{equation*}
		such that $\tilde{u}$ is Pareto critical for $\tilde{J}$.
	\end{lemma}
	\begin{proof}
		Let
		\begin{align*}
			\tilde{\alpha} &\in \text{argmin}_{\alpha \in \Delta_k} \left( \| DJ^r(u)^\top \alpha \|_2^2 - (\alpha^\top \epsilon)^2 \right), \\
			\nu &:= DJ^r(\tilde{u})^\top \tilde{\alpha}, \\
			g(u) &:= - \left( \frac{1}{\tilde{\alpha}^\top \epsilon} \sum_{i = 1}^n \nu_i u_i \right) \epsilon, \\
			\tilde{J}(u) &:= J^r(u) + g(u).			
		\end{align*}
		Since $\tilde{u} \in P_2^r$ by assumption, we have $\| \nu \|_2 \leq \tilde{\alpha}^\top \epsilon$. Thus 
		\begin{equation*}
			\| \nabla \tilde{J}_i(u) - \nabla J^r_i(u) \|_2 = \| \nabla g_i(u) \|_2 = \frac{\epsilon_i}{\tilde{\alpha}^\top \epsilon} \| \nu \|_2 \leq \epsilon_i \quad \forall u \in \mathbb{R}^n \text{ and } \forall i \in \{1,...,k\},
		\end{equation*}
		and
		\begin{align*}
			D \tilde{J}(\tilde{u})^\top \tilde{\alpha} = \nu + \sum_{i = 1}^k \tilde{\alpha}_i \nabla g_i(\tilde{u}) = \nu - \sum_{i = 1}^k \tilde{\alpha}_i \frac{\epsilon_i}{\tilde{\alpha}^\top \epsilon} \nu = 0,
		\end{align*}
		which proves the lemma.
	\end{proof}
		
	Lemma \ref{p5lem:P2_tight} shows that for each point $\tilde{u}$ in $P_2^r$, there is an objective function satisfying the error bounds \eqref{p5eq:error_bounds} for which $\tilde{u}$ is Pareto critical. As a result, $P_2^r$ is the tightest superset of $P_c$ we can hope for if we only have the estimates in \eqref{p5eq:error_bounds}. The following example shows both supersets for a simple MOP (cf.~\cite{PD2017}).	
	
	\begin{example} \label{p5example:simple}
		Let 
		\begin{equation*} \label{p5eq:MOP_simple_example}
			J^r : \mathbb{R}^2 \rightarrow \mathbb{R}^2, \ u \mapsto
			\begin{pmatrix}
				(u_1 - 1)^2 + (u_2 - 1)^4 \\
				(u_1 + 1)^2 + (u_2 + 1)^2
			\end{pmatrix}.
		\end{equation*}
		We consider the two error bounds $\epsilon^1 = (0.2, 0.05)^\top$ and $\epsilon^2 = (0, 0.2)^\top$. The corresponding supersets $P_1^r$ and $P_2^r$ are shown in Figure \ref{p5fig:simple_example}.
		
		\begin{figure}[h!] 
			\parbox[b]{0.49\textwidth}{
				\centering 
				\includegraphics[width=0.425\textwidth]{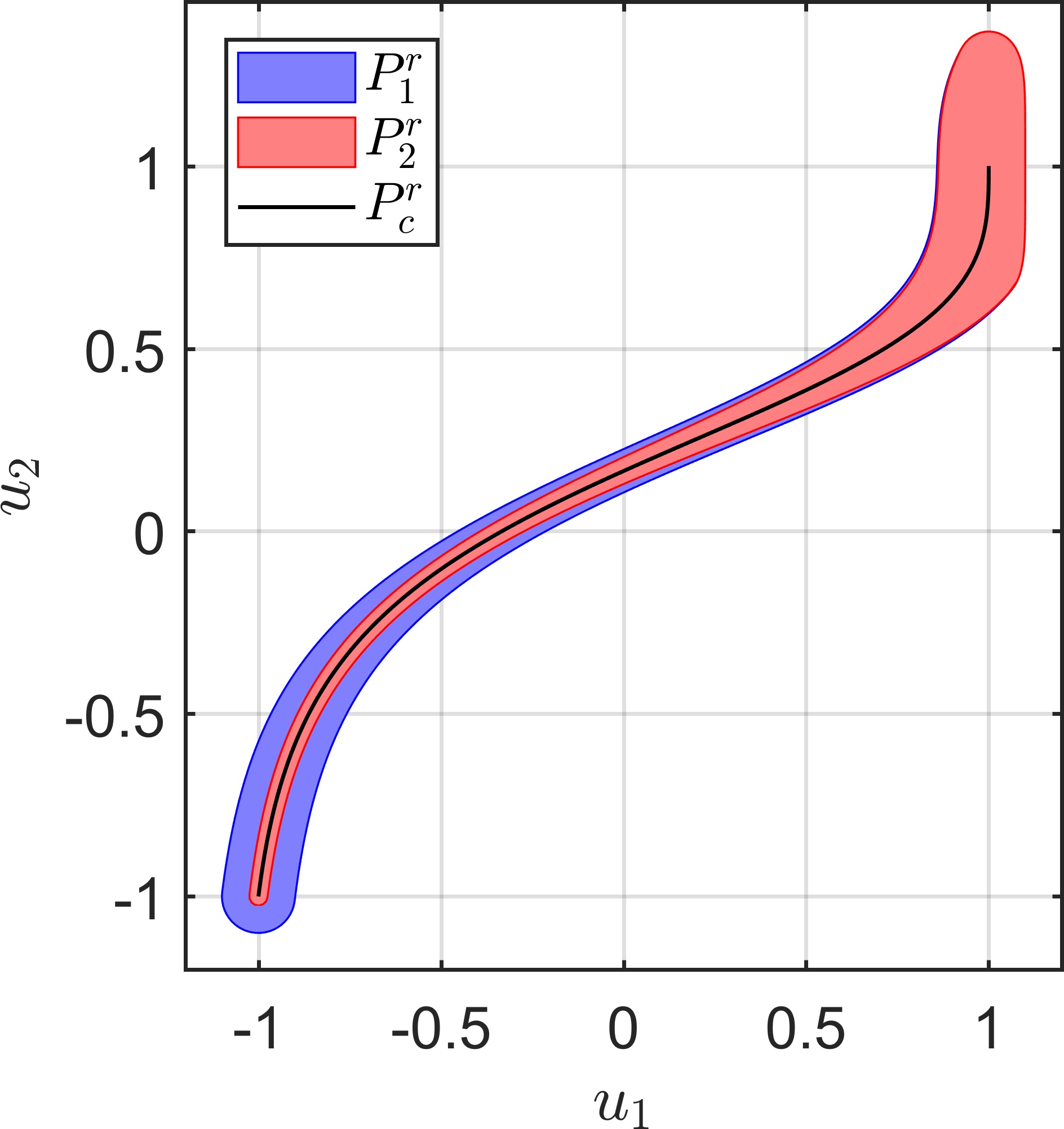}\\
				\textbf{(a) $\epsilon = \epsilon^1 = (0.2, 0.05)^\top$}
			}
			\parbox[b]{0.49\textwidth}{
				\centering 
				\includegraphics[width=0.425\textwidth]{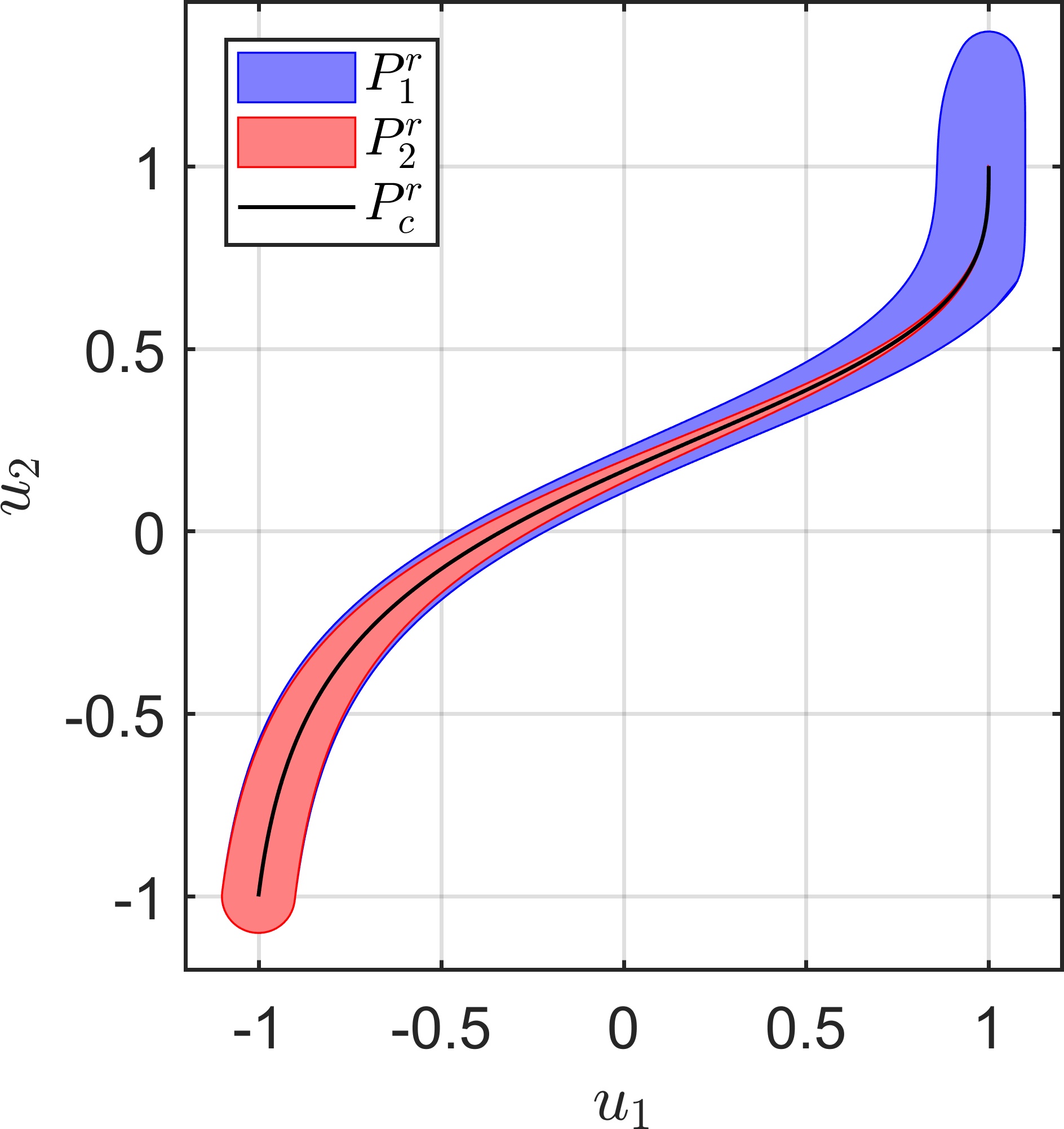}\\
				\textbf{(b) $\epsilon = \epsilon^2 = (0, 0.2)^\top$} 
			}
			\caption{$P_1^r$ and $P_2^r$ for different error bounds $\epsilon$}
			\label{p5fig:simple_example}
		\end{figure}			
		
		As $\| \epsilon_1 \|_\infty = \| \epsilon_2 \|_\infty = 0.2$, $P_1^r$ is identical for both error bounds. Considering each component of $J^r$ individually, the critical points of $J_1^r$ and $J_2^r$ are located at $u^1 = (1,1)^\top$ and $u^2 = (-1,-1)^\top$, respectively. For $P_2^r$, we see that the difference between $P_c^r$ and $P_2^r$ becomes smaller the closer we get to the critical point of the objective function with the smaller error bound. This can be expected, as the influence (or weight) of $\nabla J^r_i(u)$ in the KKT conditions \eqref{p5eq:KKT} becomes larger the closer $u$ is to $u^i$.	In particular, in Figure \ref{p5fig:simple_example}(b), the difference between $P_2^r$ and $P_c^r$ at $(1,1)^\top$ becomes zero, as $\epsilon^2_1 = 0$. \hfill$\Diamond$
	\end{example}	

	If we set $\epsilon_i = \| \epsilon \|_\infty$ for all $i \in \{1,...,k\}$, then $P_1^r = P_2^r$. Thus, we will from now on only consider $P_2^r$. As shown in the previous example, the ``dimension'' of $P_2^r$ is higher than the ``dimension'' of $P_c^r$. More precisely, $P_2^r$ contains the closure of an open subset of $\mathbb{R}^n$, which is shown in the following lemma:
	
	\begin{lemma} \label{p5lem:dim_P2}
		Let $\nabla J^r_i$ be continuous for all $i \in \{1,...,k\}$. Let
		\begin{equation*}
			A := \left\{ u \in \mathbb{R}^n : \min_{\alpha \in \Delta_k} \left( \| DJ^r(u)^\top \alpha \|_2^2 - (\alpha^\top \epsilon)^2 \right) < 0 \right\}.
		\end{equation*}
		Then
		\begin{enumerate}
			\item[(a)] $P_2^r$ is closed. In particular, $\overline{A} \subseteq P_2^r$.
			\item[(b)] $A$ is open.
		\end{enumerate}				
	\end{lemma}
	\begin{proof}
		(a) The case $P_2^r = \emptyset$ is trivial, so we assume that $P_2^r \neq \emptyset$. Let $\bar{u} \in \overline{P_2^r}$. Then there is a sequence $(u^i)_i \in P_2^r$ with $\lim_{i \rightarrow \infty} u^i = \bar{u}$. Consider the sequence $(\alpha^i)_i \in \Delta_k$ with
		\begin{equation*}
			\alpha^i \in \text{argmin}_{\alpha \in \Delta_k} \left( \| DJ^r(u^i)^\top \alpha \|_2^2 - (\alpha^\top \epsilon)^2 \right).
		\end{equation*}
		By compactness of $\Delta_k$, we can assume w.l.o.g.~that there is some $\bar{\alpha} \in \Delta_k$ with $\lim_{i \rightarrow \infty} \alpha^i = \bar{\alpha}$. Let
		\begin{equation*}
			\Psi : \mathbb{R}^n \times \Delta_k \rightarrow \mathbb{R}, \quad (u,\alpha) \mapsto \| DJ^r(u)^\top \alpha \|_2^2 - (\alpha^\top \epsilon)^2.
		\end{equation*}
		By our assumption, $\Psi$ is continuous. From $\Psi(u^i,\alpha^i) < 0$ for all $i \in \mathbb{N}$ it follows that $\Psi(\bar{u},\bar{\alpha}) \leq 0$, which yields $\bar{u} \in P_2^r$. \\
		(b) The case $A = \emptyset$ is again trivial such that we assume $A \neq \emptyset$. Let $\bar{u} \in A$ with 
		\begin{equation*}
			\bar{\alpha} \in \text{argmin}_{\alpha \in \Delta_k} \left( \| DJ^r(\bar{u})^\top \alpha \|_2^2 - (\alpha^\top \epsilon)^2 \right).
		\end{equation*}
		Let $\psi : \mathbb{R}^n \rightarrow \mathbb{R}$, $u \mapsto \| DJ^r(u)^\top \bar{\alpha} \|_2^2 - (\bar{\alpha}^\top \epsilon)^2$. Then $\psi(\bar{u}) < 0$ and by our assumption, $\psi$ is continuous. Therefore, there is some open set $U \subseteq \mathbb{R}^n$ with $\bar{u} \in U$ such that $\psi(u) < 0$ for all $u \in U$. Since 
		\begin{equation*}
			\min_{\alpha \in \Delta_k} \left( \| DJ^r(u)^\top \alpha \|_2^2 - (\alpha^\top \epsilon)^2 \right) \leq \psi(u) < 0 \quad \forall u \in U
		\end{equation*}
		we have $U \subseteq A$ such that $A$ is open.
	\end{proof}			

	We will now present two strategies for the numerical computation of $P^r_2$. Analogously to the case with exact gradients, we will approximate $P_2^r$ via the box covering
	\begin{equation*}
		\mathcal{B}_c^r(r) := \{ B \in \mathcal{B}(r) : B \cap P_2^r \neq \emptyset \}.
	\end{equation*}
		
	\subsubsection{Strategy 1}
	The idea of our first method is to mimic the exact continuation method to calculate $\mathcal{B}_c^r$. For this, there are mainly two modifications we have to make:
	\begin{enumerate}
		\item By Lemma \ref{p5lem:dim_P2}, $P_2^r$ is not a lower-dimensional object in $\mathbb{R}^n$, so it makes no sense to use tangent information to find first-order candidates as in \eqref{p5eq:exact_tangent_boxes}. Instead, we have to consider all neighboring boxes.
		\item The problem \eqref{p5eq:KKT_box_test} has to be replaced by a problem that checks the defining inequality of $P_2^r$.
	\end{enumerate}
		
	As a replacement for \eqref{p5eq:KKT_box_test}, we consider the following problem:
	
	\begin{align} \label{p5eq:KKTred_box_test}
		\min_{u \in B, \alpha \in \Delta_k} & \| DJ(u)^\top \alpha \|_2^2 - (\alpha^\top \epsilon)^2. \tag{$\epsilon$PC-Box}
	\end{align}		

	Let $\theta_\epsilon(B)$ be the optimal value of this problem. Note that $\theta_\epsilon(B) < 0$ is sufficient to verify that a box $B$ contains part of $P_2^r$. As a result, we do not need to solve \eqref{p5eq:KKTred_box_test} exactly. For example, when using an iterative method for the solution of \eqref{p5eq:KKTred_box_test}, we can stop when the function value is negative. The above mentioned changes yield Algorithm \ref{p5algo:boxcon_eps}.	
		
	\begin{algorithm} 
		\caption{Strategy 1: Box-Continuation Algorithm with Inexact Gradients}
		\label{p5algo:boxcon_eps}
		\begin{algorithmic}[1] 
			\item[] \hspace{-\algorithmicindent} Given: Radius $r > 0$ of boxes.
			\State Choose an initial point $u_0 \in P_2^r$ and initialize $\mathcal{B} = \{ B(u_0,r) \}$ and a queue $Q = \{ u_0 \}$.
			\While{$Q \neq \emptyset$}
				\State Remove the first element $\bar{u}$ from $Q$.
				\For{$B' \in N(\bar{u},r) \setminus \mathcal{B}$}
					\State Solve \eqref{p5eq:KKTred_box_test} for $B'$. Let $\theta_\epsilon(B')$ be the optimal value and $(u',\alpha')$ be
					\item[] \hspace{\algorithmicindent}\hspace{\algorithmicindent}the solution. 
					\If{$\theta_\epsilon(B') \leq 0$}
						\State Add $u'$ to $Q$ and $B'$ to $\mathcal{B}$.
					\EndIf								
				\EndFor
			\EndWhile
		\end{algorithmic} 
	\end{algorithm}			
		
	Due to the loss of low-dimensionality of $P_2^r$, the formulation of the continuation method becomes much simpler. As a consequence, it is straightforward to show that Algorithm \ref{p5algo:boxcon_eps} yields the desired covering $\mathcal{B}_c^r(r)$.
	
	When executing the exact continuation method directly using inexact gradients (i.e., forgetting about the inexactness) and comparing it to Algorithm \ref{p5algo:boxcon_eps} (with the same box radius), the former will generally be much faster than the latter. A suitable way to evaluate the run time is to compare the number of times Problems \eqref{p5eq:KKT_box_test} and \eqref{p5eq:KKTred_box_test} need to be solved, respectively, as they require the majority of the computing time and are equally difficult to solve. (Here, we assume that both problems are solved with equal precision.) For each box added to the collection $\mathcal{B}$ in either algorithm, one of these problems has to be solved. Consequently, the longer run time of Algorithm \ref{p5algo:boxcon_eps} is partly due to the fact that $P_2^r$ is a superset of $P_c^r$, which means that more boxes are required to cover $P_2^r$ than $P_c^r$. However, even if the error bounds $\epsilon$ are small such that $P_2^r$ and $P_c^r$ are almost equal, Algorithm \ref{p5algo:boxcon_eps} will be slower. This is due to the fact that instead of only the tangent boxes, all neighboring boxes have to be tested with \eqref{p5eq:KKTred_box_test} in each loop of Algorithm \ref{p5algo:boxcon_eps}. While this does not matter in the interior of $P_2^r$ (as all neighboring boxes are in fact in $P_2^r$ in that case), it is very inefficient at the boundary of $P_2^r$. This is the motivation for the second strategy.
	
	\subsubsection{Strategy 2}
	
	By Lemma \ref{p5lem:dim_P2}, $P_2^r$ has the same dimension as the space of variables $\mathbb{R}^n$. This means that it can be described much more efficiently by its topological boundary $\partial P_2^r$. To be more precise, $\mathbb{R}^n \setminus \partial P_2^r$ consists of different connected components that lie either completely inside or completely outside $P_2^r$. So if we know $\partial P_2^r$, we merely have to test one point of each connected component if it is contained in $P_2^r$ or not to completely determine $P_2^r$. Therefore, the idea of our second strategy is to only compute $\partial P_2^r$. 
	
	Let
	\begin{equation} \label{p5eq:def_varphi}
		\varphi : \mathbb{R}^n \rightarrow \mathbb{R}, \quad u \mapsto \min_{\alpha \in \Delta_k} \left( \| DJ^r(u)^\top \alpha \|_2^2 - (\alpha^\top \epsilon)^2 \right).
	\end{equation}
	This map is well-defined since $\Delta_k$ is compact, i.e., the minimum always exists. By Lemma \ref{p5lem:dim_P2}, we have $\partial P_2^r \subseteq \varphi^{-1}(0)$. Our goal is to compute $\varphi^{-1}(0)$ via a continuation approach. To this end, we first have to show that $\varphi$ is differentiable. We will do this by investigating the properties of the optimization problem in \eqref{p5eq:def_varphi}, i.e., of the problem	
	\begin{align} \label{p5eq:P2_problem}
		\min_{\alpha \in \mathbb{R}^k} \ & \omega(\alpha), \nonumber \\
		s.t. \quad & \sum_{i = 1}^k \alpha_i = 1, \\
		\quad & \alpha_i \geq 0 \quad \forall i \in \{1,...,k\}, \nonumber
	\end{align}		
	for
	\begin{equation*}
		\omega(\alpha) := \| DJ^r(u)^\top \alpha \|_2^2 - (\alpha^\top \epsilon)^2 = \alpha^\top (DJ^r(u) DJ^r(u)^\top - \epsilon \epsilon^\top) \alpha.
	\end{equation*}
	This leads to the following result.
	
	\begin{theorem} \label{p5thm:varphi_differentiable}
		Let $\bar{u} \in \varphi^{-1}(0)$ such that \eqref{p5eq:P2_problem} has a unique solution $\bar{\alpha} \in \Delta_k$ with $\bar{\alpha}_i > 0$ for all $i \in \{1,...,k\}$. Let \eqref{p5eq:P2_problem} be uniquely solvable in a neighborhood of $\bar{u}$. Then there is an open set $U \subseteq \mathbb{R}^n$ with $\bar{u} \in U$ such that $\varphi|_U$ is continuously differentiable.
	\end{theorem}
	\begin{proof}
		See Appendix~\ref{p5app:ProofPhiDiff}.
	\end{proof}
	
	For a standard continuation approach, we also have to show that $\varphi^{-1}(0)$ is a manifold. By the Level Set Theorem (cf.~\cite{L2012}, Corollary 5.14), to properly show that $\varphi^{-1}(0)$ is a manifold in a neighborhood of some $\bar{u} \in \varphi^{-1}(0)$, we would have to show that $D\varphi|_U(\bar{u}) \neq 0$ (cf.~\eqref{p5eq:D_phi_expl}). From the theoretical point of view, this poses a problem as there is no obvious way to achieve this. In practice however, we can test this by checking if the norm of $D\varphi|_U(\bar{u})$ is below a certain threshold. If this is the case, and if $\varphi^{-1}(0)$ is indeed not a manifold, we again have to consider all neighboring boxes as tangent boxes as in strategy 1. Otherwise, if $D\varphi|_U(\bar{u}) \neq 0$, we can compute the tangent space $T_{\bar{u}}$ of $\varphi^{-1}(0)$ at $\bar{u}$ via 
	\begin{equation*}
		T_{\bar{u}} = ker(D \varphi(\bar{u})).
	\end{equation*}
	Finally, in analogy to \eqref{p5eq:KKT_box_test} and \eqref{p5eq:KKTred_box_test}, we will use the following problem to test if a box $B$ contains part of $\partial P_2^r$:
	\begin{equation} \label{p5eq:bdryP2r_test}
		\min_{u \in B} \varphi(u)^2. \tag{$\partial \epsilon$PC-Box}
	\end{equation}		
	The resulting continuation method is presented in Algorithm \ref{p5algo:boxcon_boundary}.	
	
	\begin{remark} \label{p5rem:S2_subproblem}
		\begin{enumerate}
			\item Since for every evaluation of $\varphi$ the solution of the quadratic problem \eqref{p5eq:P2_problem} has to be computed, \eqref{p5eq:bdryP2r_test} is significantly more difficult to solve than \eqref{p5eq:KKTred_box_test}. Additionally, we are looking for the points $u$ where $\varphi(u) = 0$, i.e., where the problem \eqref{p5eq:P2_problem} is not positive definite. This increased difficulty of Strategy 2 is compensated by the fact that far fewer boxes have to be checked with \eqref{p5eq:bdryP2r_test} than with \eqref{p5eq:KKTred_box_test} in Strategy 1.
			\item When all $\epsilon_i = \bar{\epsilon}$ are equal, $\varphi(u) = -\bar{\epsilon}^2$ for all $u \in P_c^r$, i.e., $\varphi$ is constant on the Pareto critical set $P_c^r$. This means that local solvers may fail to find a minimum of $\varphi$ when the box $B$ in \eqref{p5eq:bdryP2r_test} has a nonempty intersection with $P^r$. An obvious but expensive way to circumvent this problem is to start the local solver multiple times with different initial points. Alternatively, one can use sufficient conditions for a box $B$ containing part of $\varphi^{-1}(0)$ before actually solving \eqref{p5eq:bdryP2r_test}. For example, by the intermediate value theorem, if there are two points in $B$ where $\varphi$ has different sign, we immediately know that $\varphi(u) = 0$ for some $u \in B$. (But note that for this method, we still need to find a point in $\varphi^{-1}(0) \cap B$ to be able to calculate the tangent space of $\varphi^{-1}(0)$).
			\item In practice, error bounds which are zero can cause problems for the stability of Strategy 2. For example, in Figure \ref{p5fig:simple_example}(b), the width of $P_2^r$ becomes arbitrarily small near $(1,1)^\top$. As a result, Strategy 2 may jump between different parts of the boundary and thus miss certain parts. Additionally, since the boundary of $P_2^r$ typically intersects the Pareto critical set $P_c$ in this case, \eqref{p5eq:bdryP2r_test} may be difficult to solve (as in 2.). Thus, in practice, one should use error bounds that are slightly larger than zero, even if the corresponding gradients are exact.
		\end{enumerate} \hfill$\blacksquare$
	\end{remark}
	
	\begin{algorithm} 
		\caption{Strategy 2: Boundary-Continuation Algorithm for Inexact Gradients}
		\label{p5algo:boxcon_boundary}
		\begin{algorithmic}[1] 
			\item[] \hspace{-\algorithmicindent} Given: Radius $r > 0$ of boxes.		
			\State Choose an initial point $u_0 \in \partial P_2^r$ and initialize $\mathcal{B} = \{ B(u_0,r) \}$ and a queue $Q = \{ u_0 \}$.
			\While{$Q \neq \emptyset$}
				\State Remove the first element $\bar{u}$ from $Q$.
				\State If $\left\lVert D \varphi(\bar{u}) \right\rVert_2$ is small set $T = \mathbb{R}^n$. Otherwise, compute the tangent space 
				\item[] \hspace{\algorithmicindent}$T = ker(D \varphi(\bar{u}))$.
				\item[]
				\item[] \hskip\algorithmicindent \textbf{Predictor:}
				\State Find all neighboring boxes of $B(\bar{u},r)$ that have a nonempty intersection 
				\item[] \hskip\algorithmicindent with $\bar{u} + T$ and have not been considered before, i.e., 
					\begin{equation*}
						\mathcal{B}'(\bar{u},r) = \{ B' \in \mathcal{B}(r) : B' \cap B(\bar{u},r) \neq \emptyset, \ B' \cap \bar{u} + T \neq \emptyset \} \setminus \mathcal{B}.
					\end{equation*}
				\item[] \hskip\algorithmicindent \textbf{Corrector:}
				\For{$B' \in \mathcal{B}'(\bar{u},r)$}
					\State Solve \eqref{p5eq:bdryP2r_test} for $B'$. Let $\theta(B')$ be the optimal value and $u'$ be the
					\item[]\hspace{\algorithmicindent}\hspace{\algorithmicindent}solution. 
					\If{$\theta(B') = 0$}
						\State Add $u'$ to $Q$ and $B'$ to $\mathcal{B}$.
					\EndIf
				\EndFor									
			\EndWhile
		\end{algorithmic} 
	\end{algorithm}			
	
	\subsection{Globalization approach} \label{p5sec:globalization}
	
	Note that all algorithms presented in this section so far approximate either $P_c$, $P_2^r$ or $\partial P_2^r$ by starting in an initial point $u_0$ and then locally exploring in all (tangent) directions. Thus, if the set we want to approximate is disconnected, we can only compute the connected component that contains $u_0$. In the following, we will describe how we can solve this problem, i.e., how our methods can be globalized.
	
	As mentioned earlier, an advantage of using boxes in the continuation method instead of points is the fact that it is easy to detect whether a region has already been explored. In particular, this allows us to start the continuation in multiple initial points at the same time, by simply adding all of them to the queue $Q$ in step 1 of Algorithms \ref{p5algo:boxcon_eps} or \ref{p5algo:boxcon_boundary} (and initializing the covering $\mathcal{B}$ with the corresponding boxes). As a result, to globalize our methods, we merely have to find an initial set $U_0$ of points such that the intersection of $U_0$ with each connected component is nonempty.
	
	For obtaining an initial set, we make use of the optimization problems that verify if a box contains part of the set we want to approximate, i.e., the problems \eqref{p5eq:KKT_box_test}, \eqref{p5eq:KKTred_box_test} and \eqref{p5eq:bdryP2r_test}. The idea is to consider a box covering as in \eqref{p5eq:def_boxes} with large radius $R$ and then simply test each box for relevant points using these problems.  Let $B_0$ be a compact superset of the set that we want to approximate (i.e., of $P_c$, $P_2^r$ or $\partial P_2^r$), e.g., a large outer box. For ease of notation, we assume that $B_0$ is a union of boxes in $\mathcal{B}(R)$. For the case of the Pareto critical set $P_c$, i.e., the globalization of the exact continuation method, the resulting method is presented in Algorithm \ref{p5algo:globalization}. The corresponding globalization methods for Algorithm \ref{p5algo:boxcon_eps} and \ref{p5algo:boxcon_boundary} are obtained by replacing \eqref{p5eq:KKT_box_test} in step 3 by \eqref{p5eq:KKTred_box_test} and \eqref{p5eq:bdryP2r_test}, respectively.
		
	\begin{algorithm} 
		\caption{Global Initialization}
		\label{p5algo:globalization}
		\begin{algorithmic}[1] 
			\item[] \hspace{-\algorithmicindent} Given: Outer box $B_0$, Radius $R > 0$ of boxes.
			\State Initialize $U_0 = \emptyset$.
			\For{$B \in \mathcal{B}(R)$ with $B \cap B_0 \neq \emptyset$}		
				\State Solve \eqref{p5eq:KKT_box_test} for $B$. Let $\theta(B)$ be the optimal value and $\bar{u}$ be the solution. 
				\If{$\theta(B) = 0$}
					\State Add $\bar{u}$ to $U_0$. 
				\EndIf
			\EndFor
		\end{algorithmic} 
	\end{algorithm}			
		
	The radius $R$ has to be chosen such that for each connected component, there is at least one box in our covering that only has an intersection with the desired component. In theory, $R$ can obviously become very small if two different connected components are very close to each other. In this case, Algorithm \ref{p5algo:globalization} becomes infeasible to use, as the number of boxes that have to be tested becomes too large. In practice however, the components are often sufficiently far apart such that a large radius is sufficient and only few boxes have to be considered.

	For the globalization of the exact continuation method and Algorithm \ref{p5algo:boxcon_eps}, we only have to take the non-connectivity of $P_c$ and $P_2^r$ into account. For Algorithm \ref{p5algo:boxcon_boundary}, an additional problem may arise since the boundary $\partial P_2^r$ does not necessarily need to be smooth. Non-smoothness of $\partial P_2^r$ is caused by points in which $\varphi$ is not differentiable. (By Theorem \ref{p5thm:varphi_differentiable}, these are points where the solution of \eqref{p5eq:bdryP2r_test} is not unique.) In these points, $\partial P_2^r$ does not posses a tangent space, and our method will be unable to continue. As a result, we have to ensure in the initialization of Algorithm \ref{p5algo:boxcon_boundary} that we choose an initial point in $U_0$ on each smooth component of $\partial P_2^r$. Visually, these can be thought of as the faces of $P_2^r$.
	
	We conclude this section with some remarks on the practical use of Algorithm \ref{p5algo:globalization}.
	
	\begin{remark}
		\begin{enumerate}
			\item For MOPs with a high-dimensional variable space, Algorithm \ref{p5algo:globalization} quickly becomes infeasible due to the exponential growth of the number of boxes in $\mathcal{B}(R)$. For these cases, an initialization based on points instead of boxes should be used, for example by applying methods from global optimization to modified versions of \eqref{p5eq:KKT_box_test}, \eqref{p5eq:KKTred_box_test} and \eqref{p5eq:bdryP2r_test}, where $u$ is not constrained to a box $B$.
			\item Instead of directly looping over all boxes in step 2 of Algorithm \ref{p5algo:globalization}, in some cases it might be more beneficial to first execute a few steps of the subdivision algorithm (cf.~\cite{DSH2005}) to quickly discard boxes that are far away from the Pareto critical set.
		\end{enumerate} \hfill$\blacksquare$
		
	\end{remark}
	
		\section{Multiobjective optimization of an elliptic PDE using the RB method}	
	\label{p5sec:MultiobjectiveParameterOptimizationWithRB}
	
	In this section we will present a multiobjective (parameter) optimization problem of an elliptic advection-diffusion-reaction equation and show how the reduced basis method can be applied in view of the continuation method for inexact gradients from Section \ref{p5subsec:ContinuationMethodInexact} (see Algorithms \ref{p5algo:boxcon_eps} and \ref{p5algo:boxcon_boundary}).
	
	\subsection{Multiobjective optimization of an elliptic PDE}
	\label{p5subsec:MultiobjectiveParameterOptimization}
	 
	Given a domain $\Omega \subset \mathbb{R}^d$, $d \in \{2,3\}$,	we consider the problem
	\begin{align}
		\min_{y,u} \; \mathcal{J}(y,u) := \left( \begin{array}{c} \frac{1}{2} \left\Vert y - y^1_d \right\Vert_{L^2(\Omega)}^2 \\
		\vdots \\
		\frac{1}{2} \left\Vert y - y^{k-1}_d \right\Vert_{L^2(\Omega)}^2 \\
		\frac{1}{2} \left\Vert u \right\Vert_{\mathbb{R}^{m}}^2
	\end{array} \right)  \tag{MPOP} \label{p5eq:MPOP}
	\end{align}
	s.t.
	\begin{align}    
	\begin{array}{r l l}
		- \sum_{i=1}^{m'} \kappa_i \chi_{\Omega_i}(x) \Delta y(x) + c \, b(x) \cdot \nabla y(x) + r \, y(x) & = f(x) 			&\text{for } x \in \Omega, \\
		\frac{\partial y}{\partial \eta} (x) & = 0 &\text{for } x \in \partial\Omega,
	\end{array} \tag{EPDE} \label{p5eq:EllipticPDE}
	\end{align}
	and the bilateral box constraints
	\begin{align}
	u_a \leq u \leq u_b, \tag{BC} \label{p5eq:BoxConstraints}
	\end{align}
	where  $u = (u_1,\ldots,u_{m}) = (\kappa_1,\ldots,\kappa_{m'},c,r) \in \mathbb{R}^m$ is the parameter of dimension $m := m' + 2$, $U_{\textsl{ad}} := \{ u \in \mathbb{R}^m \mid u_a \leq u \leq u_b \}$ is the admissible parameter set, and $y \in L^2(\Omega) =: H$ is the state variable. \\
	The domain $\Omega$ is divided into $m'$ pairwise disjoint subdomains $\Omega = \Omega_1 \dot\cup \ldots \dot\cup \, \Omega_{m'}$, such that $\kappa_i$ is the diffusion coefficient on $\Omega_i$. The vector field $b \in L^{\infty}(\Omega,\mathbb{R}^d)$ is the given advection, whose strength and orientation can be controlled by the parameter $c \in \mathbb{R}$. Moreover, the reaction coefficient is given by the parameter $r > 0$, and $f \in H$ is the inhomogeneity on the right-hand side of the equation. On the boundary we impose homogeneous Neumann boundary conditions. \\
	The cost functions $\mathcal{J}_1,\ldots,\mathcal{J}_{k-1} : H \times U_{\textsl{ad}} \to \mathbb{R}^k$ are of tracking type with respect to the desired states $y_d^1,\ldots,y_d^{k-1} \in H$, and the cost function $\mathcal{J}_k : H \times U_{\textsl{ad}} \to \mathbb{R}^k$ measures the parameter cost. \\
	Setting $V := H^1(\Omega)$ and using the parameter-dependent bilinear form $a(u;\cdot,\cdot): V \times V \to \mathbb{R}$ defined by 
	\begin{align*}
	a(u,\varphi,\psi) := & \sum_{i=1}^m u_i a_i(\varphi,\psi) \\
	:= & \sum_{i=1}^{m'} \kappa_i \int_{\Omega_i} \nabla \varphi(x) \cdot \nabla \psi(x) \, dx + c \int_\Omega b(x) \cdot \nabla \varphi(x) \psi(x) \, dx \\ 
	& + r \int_{\Omega} \varphi(x) \psi(x) \, dx,
	\end{align*}
	for all $u \in U_{\textsl{ad}}$ and $\varphi,\psi \in V$, and the linear functional $F: V \to \mathbb{R}$ given by $F(\varphi) := \langle f , \varphi \rangle_H$ for all $\varphi \in V$,		
	we can write \eqref{p5eq:EllipticPDE} in its weak formulation as: Find $y \in V$ such that
	\begin{align}
	a(u;y,\varphi) = F(\varphi) \quad \text{for all } \varphi \in V \label{p5eq:WeakFormulationPDE}
	\end{align}
	is satisfied.
	It is possible to show the unique solvability of \eqref{p5eq:WeakFormulationPDE} under some conditions on the parameter $u$.
	
	\begin{theorem}
	\label{p5thm:UniqueSolvabilityPDE}
	There are $\kappa_{\textsl{min}} \in (0,\infty)^{m'}$, $c_{\textsl{min}}, c_{\textsl{max}} \in \mathbb{R}$ with $c_{\textsl{min}} < c_{\textsl{max}}$ and $r_{\textsl{min}} \in (0,\infty)$ such that \eqref{p5eq:WeakFormulationPDE} has a unique solution $y(u) \in V$ for every parameter $u = \left(\kappa,c,r\right) \in \mathbb{R}^m$ with $\kappa > \kappa_{\textsl{min}}$, $c_{\textsl{min}} < c < c_{\textsl{max}}$ and $r > r_{\textsl{min}}$.
	\end{theorem}
	
	\begin{proof}
	It is straightforward to show that for all parameters $u \in \mathbb{R}^m$ the bilinear form $a(u;\cdot,\cdot)$ and the linear functional $F$ are continuous, and that there are $\kappa_{\textsl{min}} \in (0,\infty)^{m'}$, $c_{\textsl{min}}, c_{\textsl{max}} \in \mathbb{R}$ with $c_{\textsl{min}} < c_{\textsl{max}}$ and $r_{\textsl{min}} \in (0,\infty)$ such that $a(u;\cdot,\cdot)$ is coercive for all $u = \left(\kappa,c,r\right) \in \mathbb{R}^m$ with $\kappa > \kappa_{\textsl{min}}$, $c_{\textsl{min}} < c < c_{\textsl{max}}$ and $r > r_{\textsl{min}}$. Now the Lax-Milgram Theorem can be applied to show the unique solvability of \eqref{p5eq:WeakFormulationPDE}.
	\end{proof}
	
	With Theorem \ref{p5thm:UniqueSolvabilityPDE} in mind we can introduce the solution operator of the elliptic PDE.
	 
	\begin{definition}
	\label{p5def:SolutionOperatorStateEquation}
	Define the set $U_{\textsl{eq}} := (\kappa_{\textsl{min}},\infty) \times (c_{\textsl{min}}, c_{\textsl{max}}) \times (r_{\textsl{min}},\infty)$ with the constants from Theorem \ref{p5thm:UniqueSolvabilityPDE}. Let $\mathcal{S}: U_{\textsl{eq}} \to V \hookrightarrow H$ be defined as the solution operator of \eqref{p5eq:WeakFormulationPDE}, i.e., the function $y := \mathcal{S}(u)$ solves the weak formulation \eqref{p5eq:WeakFormulationPDE} for any parameter $u \in U_{\textsl{eq}}$.
	\end{definition}
	
	\begin{remark}
	In the following we suppose that it holds $U_{\textsl{ad}} \subset U_{\textsl{eq}}$. \hfill$\blacksquare$
	\end{remark}
		
	Using the explicit dependence of the state $y$ on the parameter $u$ for all $u \in U_{\textsl{ad}}$, the essential cost functions $J_1,\ldots,J_k: U_{\textsl{ad}} \to \mathbb{R}$ can be defined.
	
	\begin{definition}
	\label{p5def:ReducedCostFunctions}
	For any $i \in \{1,\ldots,k\}$ let the essential cost function $J_i:U_{\textsl{ad}} \to \mathbb{R}$ be given by $J_i(u) := \mathcal{J}_i(\mathcal{S}(u),u)$ for all $u \in U_{\textsl{ad}}$.
	\end{definition}
	
	For applying the continuation method from Section \ref{p5sec:ContinuationMethodWithInexactness}, which is based on Theorem \ref{p5thm:hillermeier}, to solve this multiobjective parameter optimization problem, the cost functions $J_1,\ldots,J_k$ need to be twice continuously differentiable. This is the statement of the next lemma.
	
	\begin{lemma}
	\label{p5lem:CostFunctionsTwiceContDiff}
	The cost functions $J_1,\ldots,J_k$ are twice continuously differentiable.
	\end{lemma}
	
	\begin{proof}
	It is clear that the cost function $J_k$ is twice continuously differentiable. Furthermore, it is possible to show that the solution operator $\mathcal{S}$ of \eqref{p5eq:WeakFormulationPDE} is twice continuously differentiable (this can be shown by rewriting \eqref{p5eq:WeakFormulationPDE} in the form $e(y,u) = 0$ and then using the implicit function theorem, cf.~\cite[Section 1.6]{Hinze2009}). From this it immediately follows that the cost functions $J_1,\ldots,J_{k-1}$ are twice continuously differentiable as well.
	\end{proof}
	
	For later use, we need an explicit formula for the gradients $\nabla J_1,\ldots,\nabla J_k$. Therefore, we introduce the so-called adjoint equation for all $i \in \{1,\ldots,k-1\}$: Find $p \in V$ such that it holds
	\begin{align}
	a(u;\varphi,p) = \langle y_d^i - \mathcal{S}(u), \varphi \rangle_H \quad \text{for all } \varphi \in V. \label{p5eq:AdjointEquation}
	\end{align}
	With the same arguments as in Theorem \ref{p5thm:UniqueSolvabilityPDE} it is possible to show that \eqref{p5eq:AdjointEquation} has a unique solution for all $u \in U_{\textsl{eq}}$.
	
	\begin{definition}
	\label{p5def:SolutionOperatorAdjointEquation}
	Denote by $\mathcal{A}_i: U_{\textsl{eq}} \to V \hookrightarrow H$ the solution operator of the adjoint equation \eqref{p5eq:AdjointEquation} for all $i \in \{1,\ldots,k-1\}$.
	\end{definition}
	
	Now a small computation shows that 
	\begin{align*}
	J_i'(u) h = \langle \mathcal{S}(u) - y_d^i , \mathcal{S}'(u)h \rangle_H = \partial_u a(u;\mathcal{S}(u),\mathcal{A}_i(u))h,
	\end{align*}
	which yields
	\begin{align}
	\nabla J_i(u) = \left( \begin{array}{c}
									\partial_{u} a(u;\mathcal{S}(u),\mathcal{A}_i(u)) e_1 \\
									\vdots \\
									\partial_{u} a(u;\mathcal{S}(u),\mathcal{A}_i(u)) e_m
								 \end{array} \right) = \left( \begin{array}{c}
						a_1(\mathcal{S}(u),\mathcal{A}_i(u)) \\
									\vdots \\
						a_m(\mathcal{S}(u),\mathcal{A}_i(u))
								 \end{array} \right) \label{p5eq:RepresentationGradient}
	\end{align}
for all $i \in \{1,\ldots,k-1\}$. Lastly, it is obvious that $\nabla J_k(u) = u$. 

	\subsection{The reduced basis method}
	\label{p5subsec:RB}
	For computing the Pareto critical set of the problem \eqref{p5eq:MPOP} by the exact continuation method introduced in Section \ref{p5subsec:ExactContinuation}, the problem \eqref{p5eq:KKT_box_test} has to be solved numerous times. However, already one gradient evaluation of all cost functions $\nabla J_1(u),\ldots,\nabla J_k(u)$ involves the solution of one state and $k-1$ adjoint equations. Thus, using a finite element discretization for the weak formulations \eqref{p5eq:WeakFormulationPDE} and \eqref{p5eq:AdjointEquation}, which leads to large linear equation systems, is numerically very costly and time consuming. Therefore, the use of \textit{reduced-order modelling} (ROM) is a common tool to lower the computational costs. \\ 
The idea of ROM is to use a low-dimensional subspace $V^r \subset V$ as a surrogate for the infinite-dimensional space $V$ in the weak formulations \eqref{p5eq:WeakFormulationPDE} and \eqref{p5eq:AdjointEquation}. 	
Given a finite-dimensional reduced-order space $V^r \subset V$, the reduced-order state equation reads:	Find $y^r \in V^r$ such that
	\begin{align}
	a(u;y^r,\varphi) = F(\varphi) \quad \text{for all } \varphi \in V^r \label{p5eq:WeakFormulationRBStateEquation}
	\end{align}
	is satisfied. \\
	With the same arguments as in Theorem \ref{p5thm:UniqueSolvabilityPDE} it can be shown that \eqref{p5eq:WeakFormulationRBStateEquation} has a unique solution for all $u \in U_{\textsl{eq}}$. Therefore, we can follow the procedure of Section \ref{p5subsec:MultiobjectiveParameterOptimization} and introduce the solution operator $\mathcal{S}^r: U_{\textsl{eq}} \to V^r \subset V \hookrightarrow H$ of the ROM state equation \eqref{p5eq:WeakFormulationRBStateEquation} and consequently the ROM essential cost functions $J_1^r,\ldots,J_k^r$, which are defined by $J_i^r(u) := \mathcal{J}_i(\mathcal{S}^r(u),u)$ for all $u \in U_{\textsl{ad}}$ and all $i \in \{1,\ldots,k\}$. Again, it can be shown that the functions $J_1^r,\ldots,J_k^r$ are twice continuously differentiable so that they fit into the framework of Theorem \ref{p5thm:hillermeier}. The gradient of the cost functions can also be displayed by the reduced-order adjoint equations
	\begin{align}
	a(u;\varphi,p^r) = \langle y_d^i - \mathcal{S}^r(u), \varphi \rangle_H \quad \text{for all } \varphi \in V^r, \label{p5eq:RBAdjointEquation}
	\end{align}
	for all $i \in \{1,\ldots,k-1\}$, whose solution operator we denote by $\mathcal{A}_i^r : U_{\textsl{eq}} \to V^r \subset V \hookrightarrow H$. With this definition it holds 
	\begin{align}
	\nabla J_i^r(u) = \left( \begin{array}{c}
									\partial_{u} a(u;\mathcal{S}^r(u),\mathcal{A}^r_i(u)) e_1 \\
									\vdots \\
									\partial_{u} a(u;\mathcal{S}^r(u),\mathcal{A}^r_i(u)) e_m
								 \end{array} \right) = \left( \begin{array}{c}
						a_1(\mathcal{S}^r(u),\mathcal{A}^r_i(u)) \\
									\vdots \\
						a_m(\mathcal{S}^r(u),\mathcal{A}^r_i(u))
								 \end{array} \right) \label{p5eq:RepresentationGradientRB}
	\end{align}
	for all $i \in \{1,\ldots,k-1\}$. Moreover, we have $\nabla J_k^r(u) = u = \nabla J_k(u)$. \vspace{\baselineskip}
	
	In this paper we use a particular model-order reduction technique, namely the \textit{reduced basis} (RB) method (see e.g. \cite{Rozza2007,Hesthaven2016,Quarteroni2016}). In the RB method the snapshot space $V^r$ is spanned by solutions of the state equation and the adjoint equations to different parameter values $u \in U_{\textsl{ad}}$. The reduced basis is then given by an orthonormal basis $(\Phi_1,\ldots,\Phi_N)$ of the space $V^r$. \\ 
By using the RB method we introduce an error in the state equation, which transfers to the cost functions, its gradients and eventually to the Pareto critical set, which we want to compute. In Section \ref{p5subsec:ContinuationMethodInexact} two strategies were presented to deal with the inflicted inexactness in the gradients of the multiobjective optimization problem. Both are based on the estimates \eqref{p5eq:error_bounds} for the errors in the gradients of the cost functions. Thus, when applying the RB method we need to ensure these estimates. This is done by using the well-known greedy algorithm (cf.~\cite{Buffa2012}). Given a sufficiently fine finite parameter training set $\mathcal{P} \subset U_{\textsl{ad}}$ new solution snapshots are computed until the error in the gradients of all cost functions is smaller than the predefined error tolerance for all parameters in $\mathcal{P}$. The parameter for the new snapshots is thereby chosen as the one for which the error in the gradient is the largest. The procedure is summarized in Algorithm \ref{p5algo:GreedyAdjoint}.
	
	\begin{algorithm} 
		\caption{Greedy Algorithm}
		\label{p5algo:GreedyAdjoint}
		\begin{algorithmic}[1]
			\item[] \hspace{-\algorithmicindent} Given: Parameter set $\mathcal{P} \subset U_{\textsl{ad}}$, greedy tolerances $\varepsilon_1,\ldots,\varepsilon_k > 0$.
			\State Choose $u \in \mathcal{P}$, compute $\mathcal{S}(u)$,$\mathcal{A}_1(u),\ldots,\mathcal{A}_{k-1}(u)$.
			\State Set $V^r =  \textsl{span} \{ \mathcal{S}(u), \mathcal{A}_1(u),\ldots,\mathcal{A}_{k-1}(u) \}$ and compute the reduced basis by orthonormalization.
			\While{$\max_{u \in \mathcal{P}} \max_{i \in \{1,\ldots,k-1\}} \left\lVert \nabla J_i(u) - \nabla J_i^r(u) \right\rVert_2 > \epsilon_i$}
				\State Choose $(\bar{u},i) = \arg\max_{u \in \mathcal{P}, \, i \in \{1,\ldots,k-1\}} \left\lVert \nabla J_i(u) - \nabla J_i^r(u) \right\rVert_2$.
				\State Compute $\mathcal{S}(\bar{u})$ and $\mathcal{A}_i(\bar{u})$.
				\State Set $V^r =  \textsl{span} \left\{ V^r \cup \{ \mathcal{S}(\bar{u}), \mathcal{A}_i(\bar{u}) \} \right\}$ and compute the reduced basis by
				\item[] \hskip\algorithmicindent orthonormalization.
			\EndWhile
		\end{algorithmic} 
	\end{algorithm}
	
	\subsection{Error estimation for the gradients}
	
	In the greedy procedure in Algorithm \ref{p5algo:GreedyAdjoint}, the error between the full-order and the reduced-order gradients has to be evaluated. There are two strategies to do so.
	\begin{enumerate}
	\item The full-order gradients are computed and stored at the beginning of the greedy procedure. Therefore, in each greedy iteration, only the reduced-order gradients have to be computed and the error can be easily evaluated. Of course, this implies large computational costs at the beginning of the greedy procedure. This method is called strong greedy algorithm (cf.~\cite{Haasdonk2013,Buffa2012}).
	\item An a-posteriori error estimator for the errors in the gradient is used, which can be efficiently evaluated. This results in computational costs for the greedy algorithm, which only depend on the reduced-order dimension $N$. 
	\end{enumerate}
	To be able to follow the second strategy we introduce a rigorous a-posteriori error estimator for the error in the gradient of the cost functions. \\
	Using the gradient representations \eqref{p5eq:RepresentationGradient} and \eqref{p5eq:RepresentationGradientRB} we can write for $i \in \{1,\ldots,k-1\}$
	\begin{align*}
	\left\lVert \nabla J_i(u) - \nabla J^r_i(u) \right\rVert_2^2 = \sum_{j=1}^{m} \left| a_j(\mathcal{S}(u),\mathcal{A}_i(u)) - a_j(\mathcal{S}^r(u),\mathcal{A}^r_i(u)) \right|^2.
\end{align*}		
	Due to the bilinearity and the continuity of $a_1,\ldots,a_m$ and the triangle inequality, we can further write
	\begin{align}
& \left| a_j(\mathcal{S}(u),\mathcal{A}_i(u)) - a_j(\mathcal{S}^r(u),\mathcal{A}^r_i(u)) \right| \nonumber \\
	\leq & \left| a_j(\mathcal{S}(u) - \mathcal{S}^r(u),\mathcal{A}^r_i(u)) \right| + \left| a_j(\mathcal{S}(u) - \mathcal{S}^r(u),\mathcal{A}_i(u)-\mathcal{A}^r_i(u)) \right| \nonumber \\
	& + \left| a_j(\mathcal{S}^r(u),\mathcal{A}_i(u)-\mathcal{A}^r_i(u)) \right| \label{p5eq:GradientEstimateTriangleInequality} \\
	\leq & C_j \left( \left\Vert \mathcal{S}(u) - \mathcal{S}^r(u) \right\Vert_V \left\Vert \mathcal{A}^r_i(u) \right\Vert_V + \left\Vert \mathcal{S}(u) - \mathcal{S}^r(u) \right\Vert_V \left\Vert \mathcal{A}_i(u) - \mathcal{A}^r_i(u) \right\Vert_V \right. \nonumber \\
	& \left. \qquad + \left\Vert \mathcal{S}^r(u) \right\Vert_V \left\Vert \mathcal{A}_i(u) - \mathcal{A}^r_i(u) \right\Vert_V \right) \label{p5eq:GradientEstimateFinal}
	\end{align}
	for all $j \in \{1,\ldots,m\}$. \\
	Therefore, we need a-posteriori error estimators for the state and the adjoint equations in order to be able to estimate the approximation error induced in the gradients. To this end, we use the following well-known estimators (cf. \cite{Rozza2007}).
	\begin{align*}
	\left\Vert \mathcal{S}(u) - \mathcal{S}^r(u) \right\Vert_V & \leq \frac{\left\Vert r_\mathcal{S}(u) \right\Vert_{V'}}{\alpha(u)} =: \Delta_{\mathcal{S}}(u), \\
	\left\Vert \mathcal{A}_i(u) - \mathcal{A}^r_i(u) \right\Vert_V  & \leq \frac{\left\Vert r_{\mathcal{A}_i}(u) \right\Vert_{V'}}{\alpha(u)} + \Delta_{\mathcal{S}}(u) =: \Delta_{\mathcal{A}_i}(u),
	\end{align*}
	where the residuals $r_\mathcal{S}(u)$ and $r_{\mathcal{A}_i}(u)$ are given by 
	\begin{align*}
	\langle r_\mathcal{S}(u) , \varphi \rangle_{V',V} & := F(\varphi) - a(u;\mathcal{S}^r(u),\varphi) \quad & \text{for all } \varphi \in V, \\
	\langle r_{\mathcal{A}_i}(u) , \varphi \rangle_{V',V} & := \langle y_d^i - \mathcal{S}^r(u) , \varphi \rangle_H - a(u;\varphi,\mathcal{A}_i^r(u)) \quad & \text{for all } \varphi \in V. \\
	\end{align*}
	For methods on how to estimate $\alpha(u)$ and to evaluate the terms $\left\Vert r_\mathcal{S}(u) \right\Vert_{V'}$ and $\left\Vert r_{\mathcal{A}_i}(u) \right\Vert_{V'}$ efficiently, we refer for example to \cite{Rozza2007}.
	
	\begin{remark}
	Since $J_k = J_k^r$, the gradients of the two functions also coincide, so that the $\nabla J_k$ is approximated exactly by $\nabla J_k^r$. \hfill$\blacksquare$
	\end{remark}

	\section{Numerical results} \label{p5sec:NumericalResults}
	
	In this section we will numerically investigate the application of the continuation method presented in Section \ref{p5sec:ContinuationMethodWithInexactness} to the PDE-constrained multiobjective optimization problem using the reduced basis method in Section \ref{p5sec:MultiobjectiveParameterOptimizationWithRB}. \\
	For the discretization of the state and adjoint equations we used linear finite elements with 714 degrees of freedom.
	
	\subsection{Generation of the reduced basis}
	
	For investigating the generation of the reduced basis by the greedy algorithm in Algorithm \ref{p5algo:GreedyAdjoint}, we consider the MPOP
	\begin{align}
		 \left( \begin{array}{c}
		 J_1(u) \\
		 J_2(u)
\end{array}	\right)	  = \left( \begin{array}{c} \frac{1}{2} \left\Vert \mathcal{S}(u) - y^1_d \right\Vert_H^2 \\
		\frac{1}{2} \left\Vert u \right\Vert_{\mathbb{R}^{4}}^2
	\end{array} \right) \label{p5eq:MPOPExampleRBGeneration}
	\end{align}	
	 with $u = (\kappa_1,\kappa_2,c,r)$, $\Omega_1 = (0,1) \times (0,0.5)$, $\Omega_2 = (0,1) \times (0.5,1)$, and the admissible parameter set 
	 \[
	 U_{ad} = \{ u = (\kappa_1,\kappa_2,c,r) \in \mathbb{R}^4 \mid 0.2 \leq \kappa_i \leq 5 \, (i = 1,2), \; c = 0, \; r = 0.5 \}.
	 \]
	
	The reason for setting $c = 0$ in this example is that the coercivity constant $\alpha(u)$ of the bilinear form $a(u;\cdot,\cdot)$ is explicitly given by $\alpha(u) = \min \{ \kappa_1,\kappa_2,r \}$ for all $u \in U_{ad}$, so that we expect a good efficiency of the error estimator of both the state and adjoint equations. 
		\begin{figure}[h!] 
			\parbox[b]{0.49\textwidth}{
				\centering 
				\includegraphics[width=0.425\textwidth]{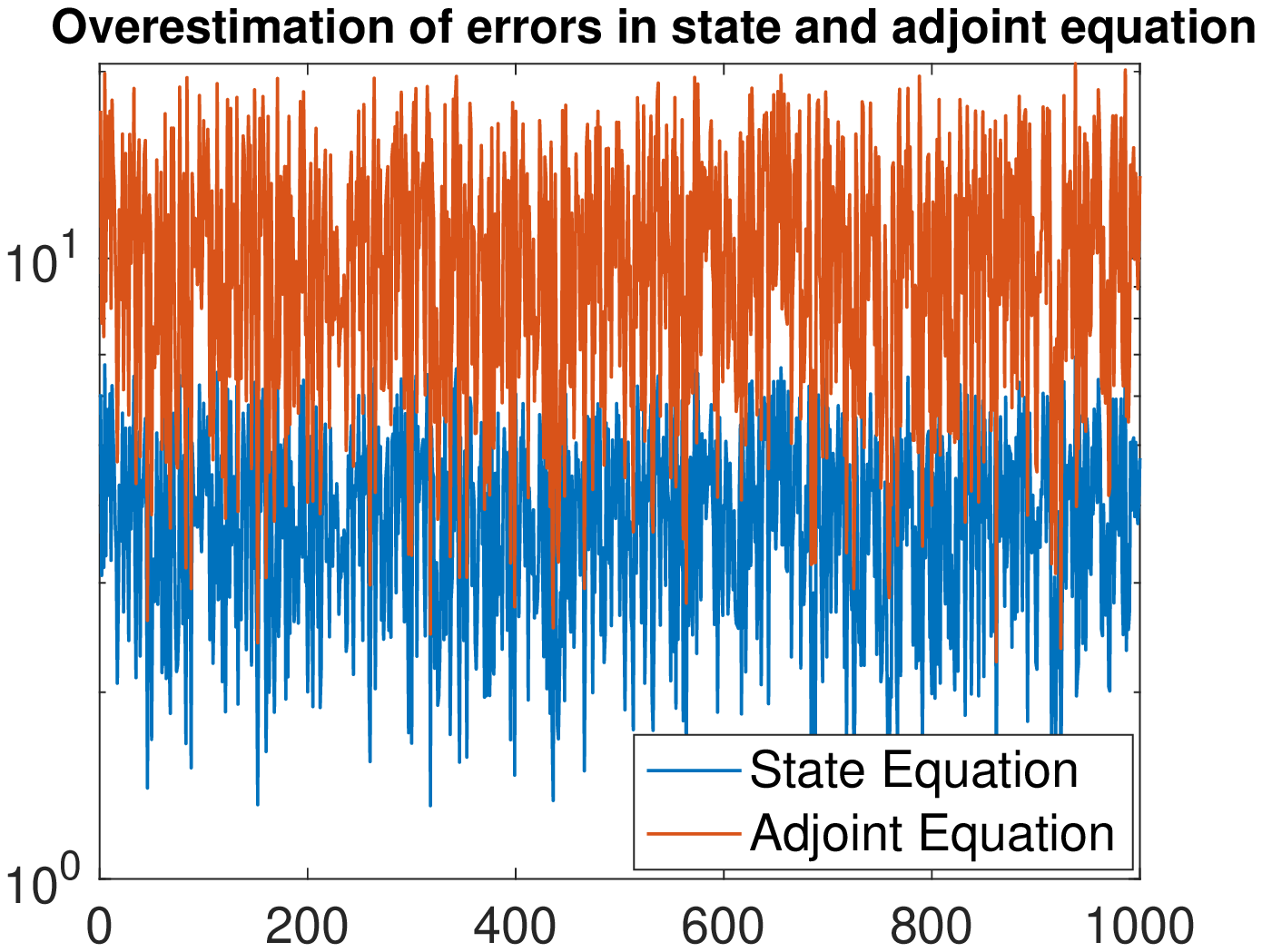}\\
				\textbf{(a) State and adjoint equation}
			}
			\parbox[b]{0.49\textwidth}{
				\centering 
				\includegraphics[width=0.425\textwidth]{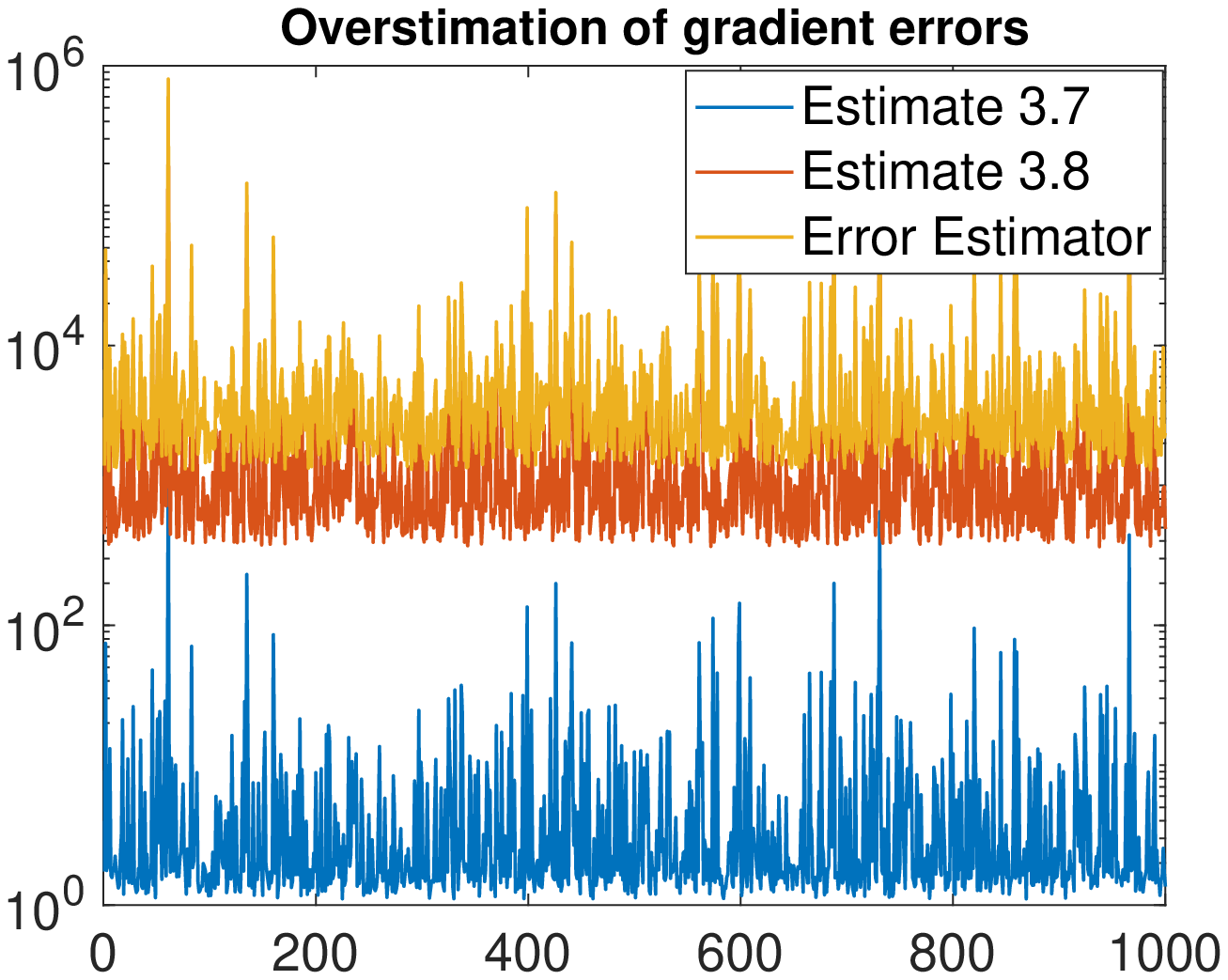}\\
				\textbf{(b) Gradient} 
			}
			\caption{Overestimations for $1000$ randomly selected parameter values}
			\label{p5fig:RBGeneration_Overestimation}
		\end{figure}	
	This is verified by the results shown in Figure \ref{p5fig:RBGeneration_Overestimation} (a), where the efficiency of the error estimator for both equations is shown for a given reduced basis for 1000 randomly chosen parameter values. However, the resulting efficiency of the error estimator for the error in the gradient is between $10^3$ and $10^6$ (see Figure \ref{p5fig:RBGeneration_Overestimation} (b)) and thus not well suited for a greedy procedure, which depends on a good error estimation. The huge overestimation of the error estimator is mainly due to the use of the triangle inequality \eqref{p5eq:GradientEstimateTriangleInequality} and the continuity estimates \eqref{p5eq:GradientEstimateFinal}, as can be seen in Figure \ref{p5fig:RBGeneration_Overestimation} (b). \\
	\begin{table}[h!]
	\caption{Number of basis functions for different error bounds}
\centering
\begin{tabular}{|l || c | c|} 
\hline
Error bound & Strong Greedy  & Error Estimate \\
 \hline
$\varepsilon = 1e-6$ & 24 & 56 \\
$\varepsilon = 1e-5$ & 20 & 50 \\
$\varepsilon = 1e-4$ & 16 & 40 \\
$\varepsilon = 1e-3$ & 12 & 32 \\
$\varepsilon = 1e-2$ & 12 & 26 \\
$\varepsilon = 1e-1$ & 10 & 20 \\
\hline
\end{tabular}
\vspace{0.2cm}
\label{p5tab:RBGenerationComparison}
\end{table}
	Compared to the strong greedy algorithm, we can see in Table \ref{p5tab:RBGenerationComparison} that this overestimation results in far more basis elements than actually needed to reach the given error bound. Since we want to investigate the influence of the error bounds in the estimate \eqref{p5eq:error_bounds} on the problem, we want that the estimate \eqref{p5eq:error_bounds} is satisfied sharply by the RB. Therefore, we will not use the error estimator to generate the basis, but instead use the strong greedy algorithm.

	\subsection{Application of the continuation methods to an MPOP}
	For the numerical investigation of the continuation method applied to a PDE-constrained multiobjective parameter optimization problem together with the use of the reduced basis method, we consider the MPOP
	\begin{align}
		 \left( \begin{array}{c}
		 J_1(u) \\
		 J_2(u) \\
		 J_3(u) \\
		 J_4(u)
\end{array}	\right)	  = \left( \begin{array}{c} \frac{1}{2} \left\Vert \mathcal{S}(u) - \mathcal{S}((0.7,0.8,0.5)) \right\Vert_H^2 \\
		\frac{1}{2} \left\Vert \mathcal{S}(u) - \mathcal{S}((2,0.5,0.5)) \right\Vert_H^2 \\
		\frac{1}{2} \left\Vert \mathcal{S}(u) - \mathcal{S}((3,-0.5,0.5)) \right\Vert_H^2 \\
		\frac{1}{2} \left\Vert u \right\Vert_{\mathbb{R}^{3}}^2
	\end{array} \right)  \label{p5eq:MPOPExampleNumerical}
	\end{align}
with $u = (\kappa,c,r)$ and 
\[
U_{ad} = \{ u = (\kappa,c,r) \in \mathbb{R}^3 \mid 0.5 \leq \kappa \leq 3, \; -1 \leq c \leq 1, \; r = 0.5 \},
\]
i.e., the reaction parameter $r$ is a constant so that we only optimize the diffusivity in the whole domain $\Omega$ and the strength and orientation of the advection field $b$. Thus, this can be seen as a problem with two parameters. 
 
As described before, the reduced basis is generated by the strong greedy Algorithm \ref{p5algo:GreedyAdjoint}, where the error bounds $\epsilon_1,\ldots,\epsilon_4$ are chosen in accordance with the estimate \eqref{p5eq:error_bounds}. As a reference, the exact solution of \eqref{p5eq:MPOPExampleNumerical} (via exact continuation and FEM discretization of the weak formulations) is shown in Figure \ref{p5fig:example_FEM_solution}.
	
	\begin{figure}[ht]
		\centering
		\includegraphics[width=0.4\textwidth]{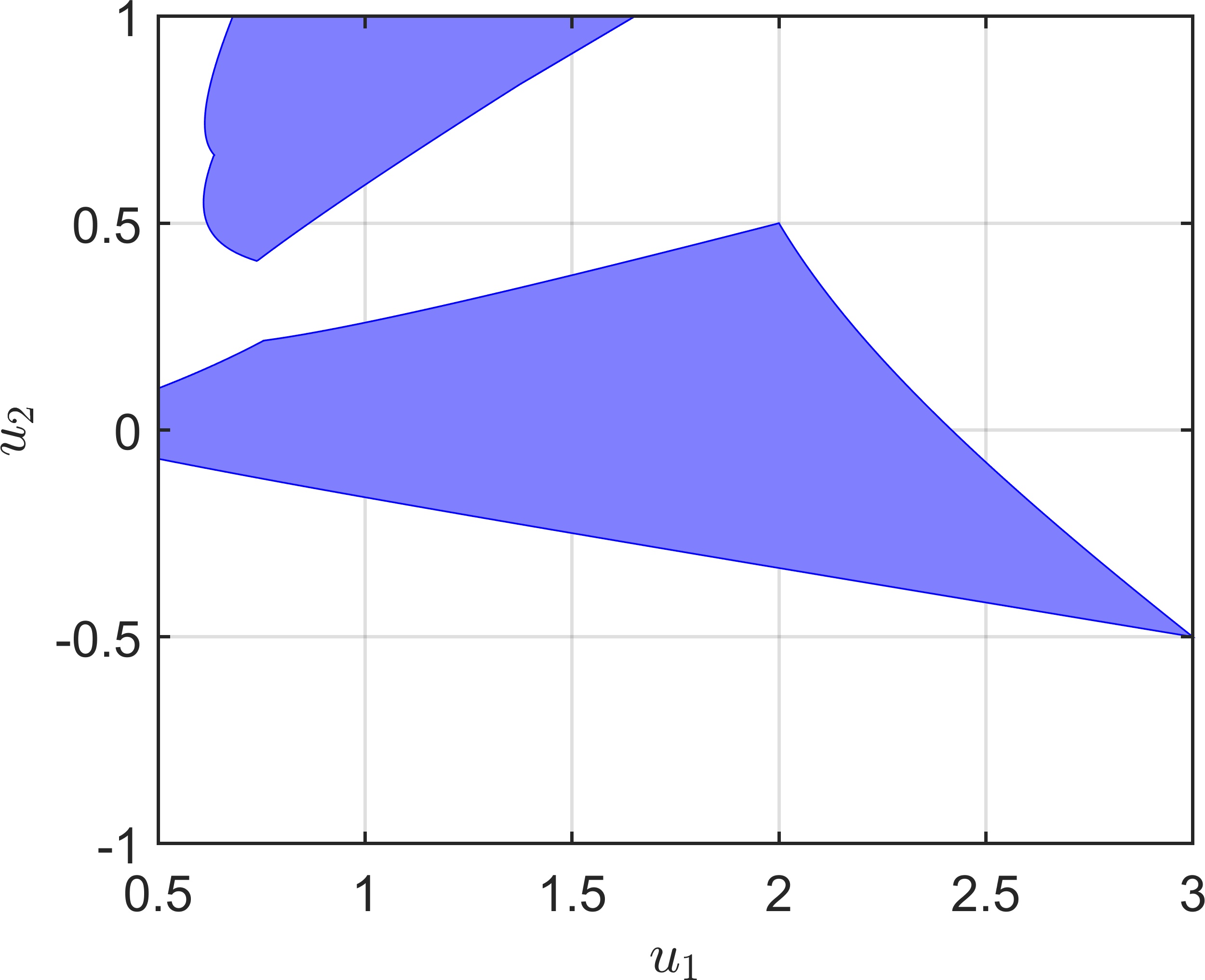}
		\caption{The Pareto critical set of \eqref{p5eq:MPOPExampleNumerical}}
		\label{p5fig:example_FEM_solution}
	\end{figure}
	
	\begin{remark}
		Since \eqref{p5eq:MPOPExampleNumerical} is constrained to a box, we have to use a constrained version of the exact continuation method (cf.~\cite{H2001}) to calculate Pareto critical points that lie on the boundaries of \eqref{p5eq:MPOPExampleNumerical}. But note that for this example, all Pareto critical points on the boundary are also Pareto critical if we ignore the constraints. In other words, for each Pareto critical point $\bar{u}$ on the boundary, there is a sequence of Pareto critical point in the interior that converges to $\bar{u}$. By continuity of $DJ$, the gradients of the (active) inequality constraints in the KKT conditions can be ignored. As a result, we can treat \eqref{p5eq:MPOPExampleNumerical} as an unconstrained problem that we only solve in a certain area. \hfill$\blacksquare$
	\end{remark}
	
	As a first test, we will compare the time needed to compute the exact solution of \eqref{p5eq:MPOPExampleNumerical} with the time needed for Strategy 1 and 2. For the error bounds we choose $\epsilon = (0.03, 0.03, 0.01, 0.01)$ and for the box radius we choose $r = \frac{3 - 0.5}{2^9} \approx 0.0049$. The results are shown in Figure \ref{p5fig:runtime_compare}.
	
	\begin{figure}[h!] 
		\parbox[b]{0.49\textwidth}{
			\centering 
			\includegraphics[width=0.425\textwidth]{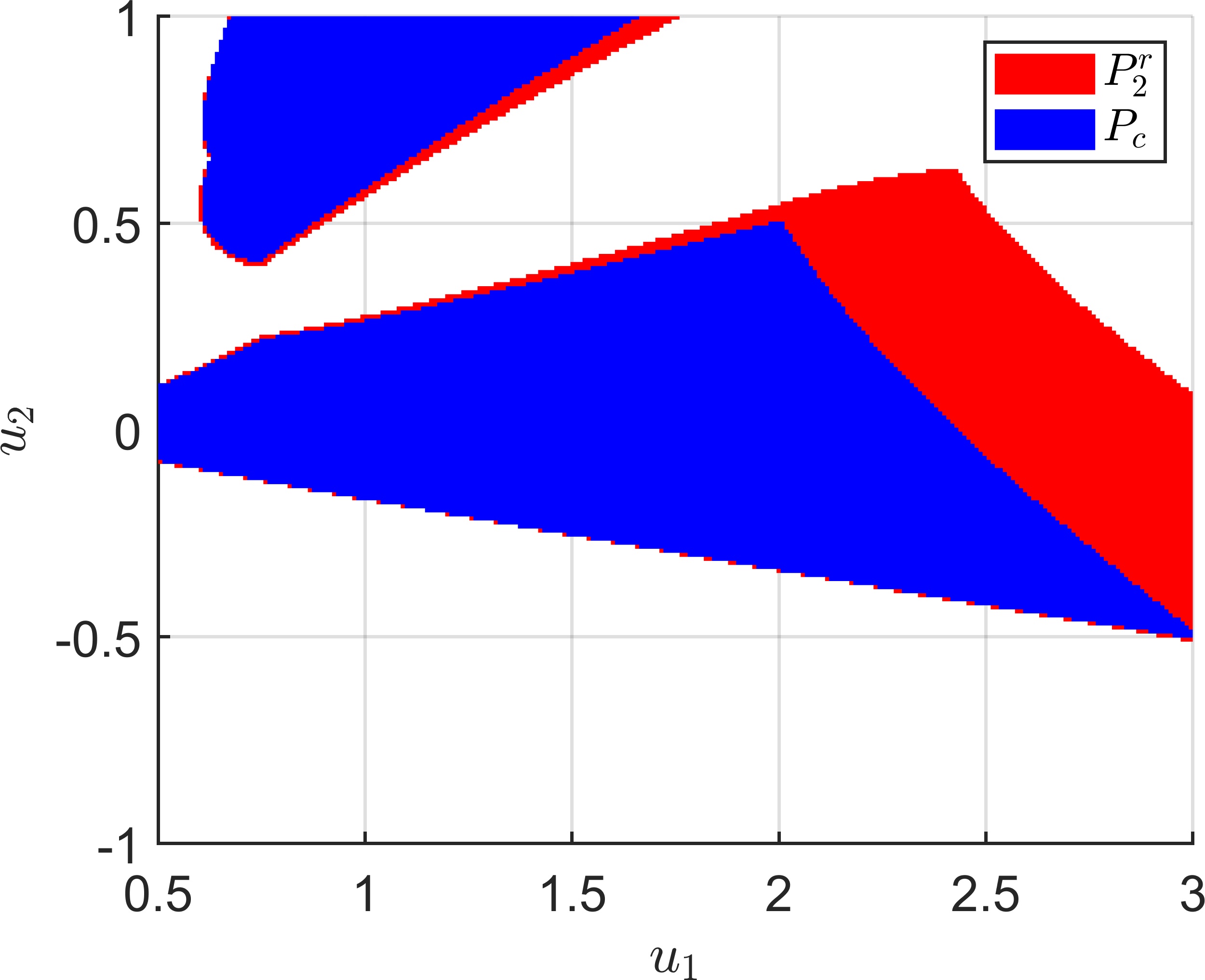}\\
		}
		\parbox[b]{0.49\textwidth}{
			\centering 
			\includegraphics[width=0.425\textwidth]{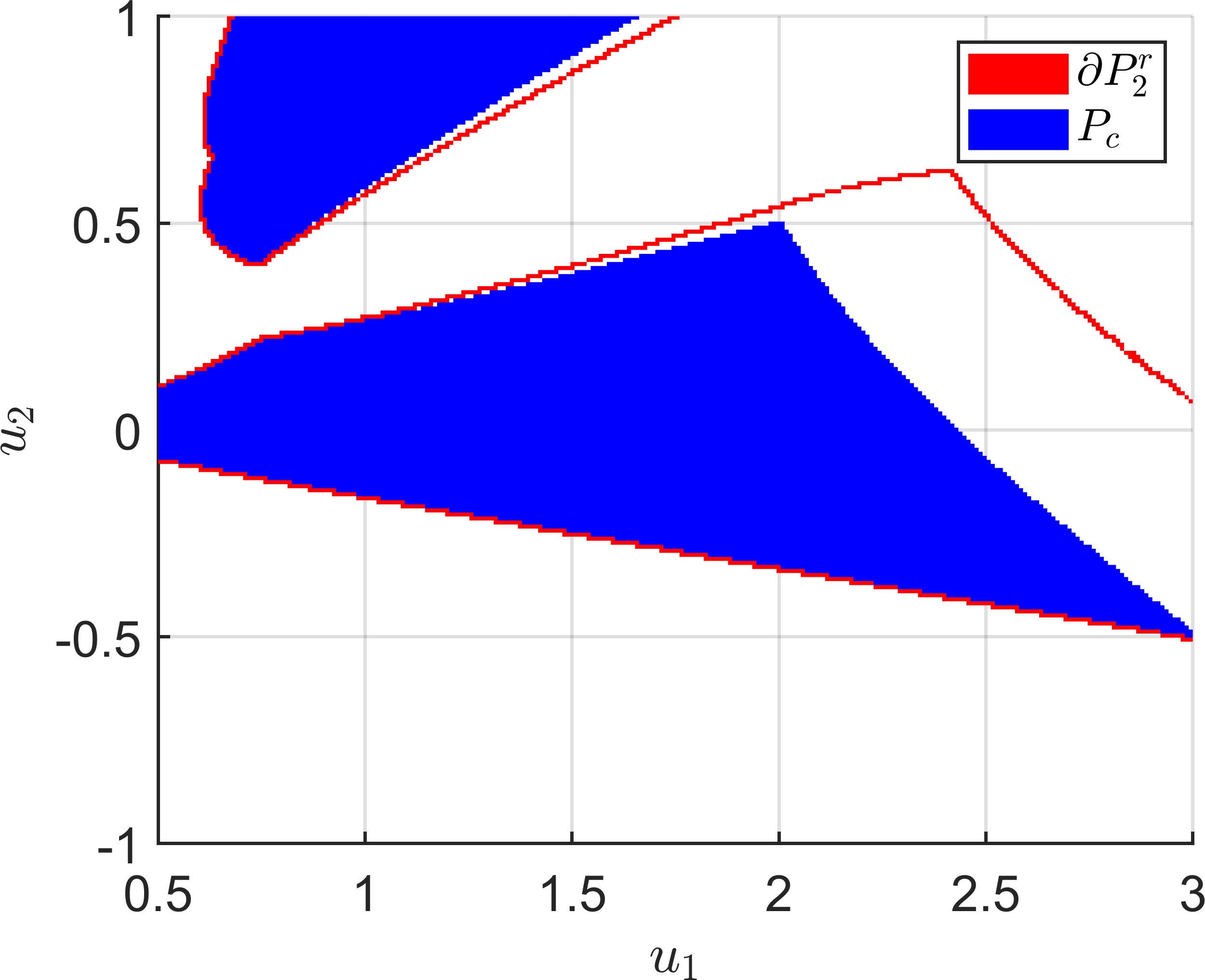}\\
		}
		\caption{Results of Strategy 1 (left) and 2 (right) for the MPOP \eqref{p5eq:MPOPExampleNumerical} with $\epsilon = (0.03,0.03,0.01,0.01)$}
		\label{p5fig:runtime_compare}
	\end{figure}	
	
	All three methods were implemented in Matlab. For the solution of the subproblems \eqref{p5eq:KKT_box_test}, \eqref{p5eq:KKTred_box_test} and \eqref{p5eq:bdryP2r_test}, the SQP-Algorithm of \verb+fmincon+ was used. (For increased stability during the continuation, each subproblem where the SQP-Algorithm found an optimal value larger than zero was restarted using the Interior-Point-Method and the Active-Set-Method of \verb+fmincon+). The runtime, number of boxes and number of subproblems needed are shown in Table \ref{p5tab:comparison}.
	\begin{table}[h]
	    \caption{Comparison of the performance of the exact continuation method, Strategy 1 and Strategy 2 for Example \eqref{p5eq:MPOPExampleNumerical}. The number of subproblems is split up in subproblems for the continuation and initialization (cf.~Section \ref{p5sec:globalization})}
		\begin{tabular}{ | l | l | l | l |}
			\hline
   			Algorithm & \# Boxes & \# Subproblems  & Runtime (in seconds) \\ \hline
			Exact cont. & $15916$ & $18721 + 25$  & $17501 s$ \\ \hline
			Strategy 1 & $21750$ & $24490 + 25$ & $1426 s$ \\ \hline
			Strategy 2 & $899$ & $1027 + 225$ & $276 s$ \\
    		\hline
    	\end{tabular}
    	\label{p5tab:comparison}
	\end{table}
	When comparing Strategy 1 and Strategy 2, we see that Strategy 2 needs about $20$ times fewer boxes and solutions of subproblems than Strategy 1. This is to be expected, since Strategy 2 only computes a covering of the boundary of $P_2^r$, i.e., of a lower dimensional set. When comparing the actual runtime, Strategy 2 is about $5$ times faster than Strategy 1, since the subproblems in Strategy 2 are more expensive to solve than the ones in Strategy 1 (cf.~Remark \ref{p5rem:S2_subproblem}). Finally, Strategy 2 is about $63$ times faster than the exact continuation method with FEM discretization, illustrating the large increase in efficiency we gain from our approach.
	
	Although it is a lot quicker to use inexact gradients from ROM instead of the exact gradients via FEM, it is important to keep in mind that our methods are computing a superset of the actual Pareto critical set. For example, in Figure \ref{p5fig:runtime_compare}, the right side of the lower connected component is only approximated poorly by $P_2^r$. Therefore, we will now investigate the influence of the error bounds $\epsilon = (\epsilon_1,\epsilon_2,\epsilon_3,\epsilon_4)$ on $P_2^r$, by applying Strategy 2 with reduced bases for different values of $\epsilon$. Note that in all our tests we set $\epsilon_4 = 0.01$, although the error in the gradient of the fourth cost function is zero for all parameters. This is done to make the solution of \eqref{p5eq:bdryP2r_test} in line 7 of Algorithm \ref{p5algo:boxcon_boundary} numerically stable (cf.~Remark \ref{p5rem:S2_subproblem}).\\
	The results of our experiment can be seen in Figure \ref{p5fig:Strategy2DifferentEpsilon}. Generally, as expected, the boundary $\partial P_2^r$ encloses the Pareto critical set $P_c$ sharper and sharper for decreasing $\epsilon$. Moreover, we observe that it is crucial to choose an $\epsilon$ which is not too large: For the value $\epsilon = (0.1,0.1,0.1,0.01)$ the shape of the boundary $\partial P_2^r$ implies that the set $P_2^r$ is connected, i.e., we lose the topological information that the Pareto critical set actually consists of two connected components. Decreasing $\epsilon$ to $\epsilon = (0.0885,0.0885,0.0885,0.01)$ we are in the limit case in which the boundary $\partial P_2^r$ touches the box constraints at around $(2.3,1)$, so that this is the approximate $\epsilon$ for which we regain the basic topological information of a disconnected Pareto critical set. \\ 
If we compare the results for $\epsilon = (0.03,0.03,0.03,0.01)$, $\epsilon = (0.03,0.03,0.01,0.01)$ and $\epsilon = (0.03,0.01,0.01,0.01)$, the influence of changing one component of $\epsilon$ becomes obvious. For $\epsilon = (0.03,0.03,0.03,0.01)$ the set $\partial P_2^r$ encloses the set $P_c$ quite sharply at the upper connected component and at the left part of the lower connected component, where the second and third component of the corresponding KKT-multipliers $\alpha$ are small. On the other hand, in the right part of the lower connected component of $P_c$, where the second and third component of the corresponding KKT-multipliers are relatively large, the deviation of $\partial P_2^r$ to $P_c$ is still large. Consequently, first reducing $\epsilon_3$ and then also $\epsilon_2$ from 0.03 to 0.01 leads to a clearly visible sharper enclosing of this part of $P_c$.  
	
	\begin{figure}[h!] 
			\parbox[b]{0.32\textwidth}{
				\includegraphics[width=0.32\textwidth]{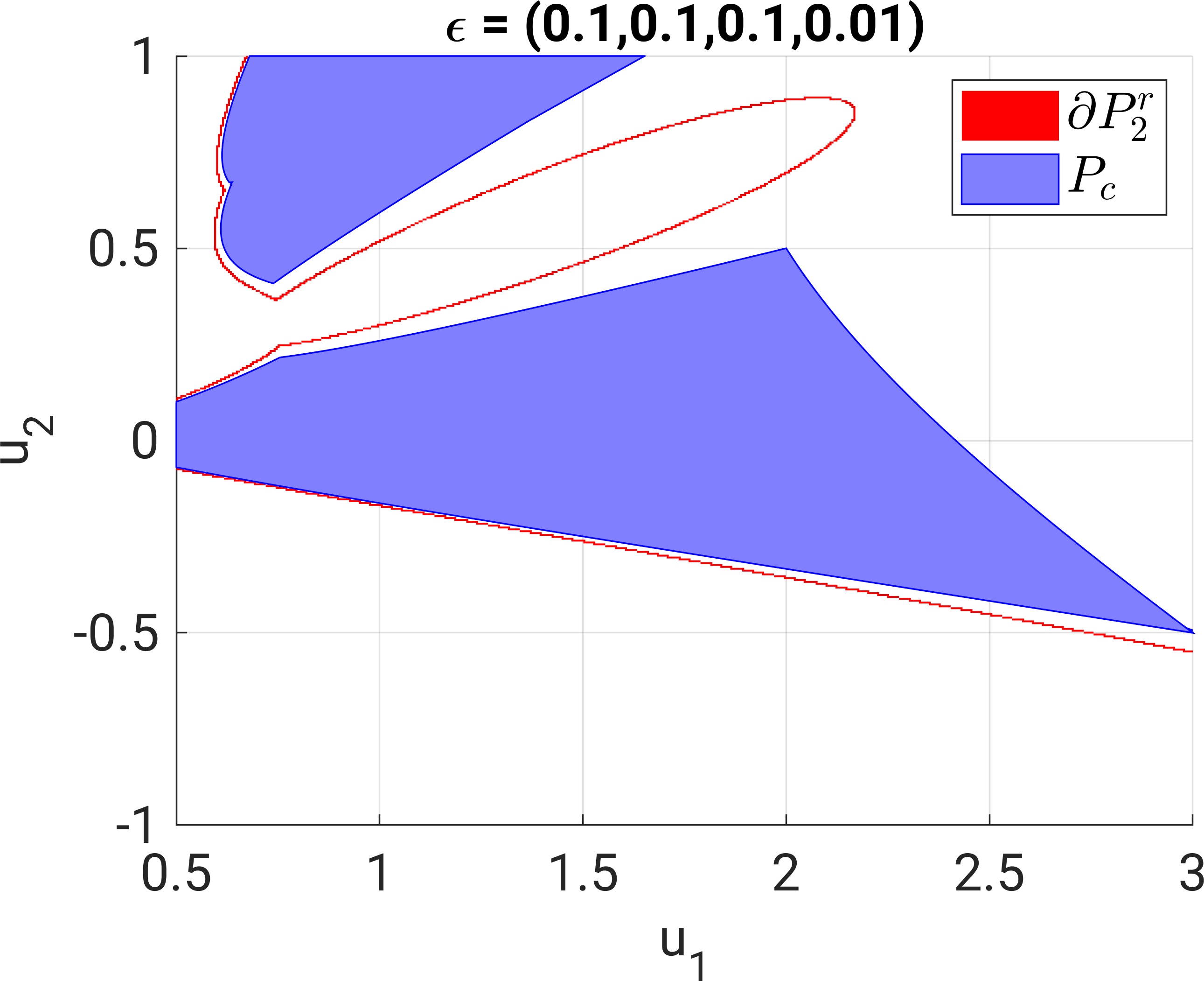}
			}
			\parbox[b]{0.32\textwidth}{

				\includegraphics[width=0.32\textwidth]{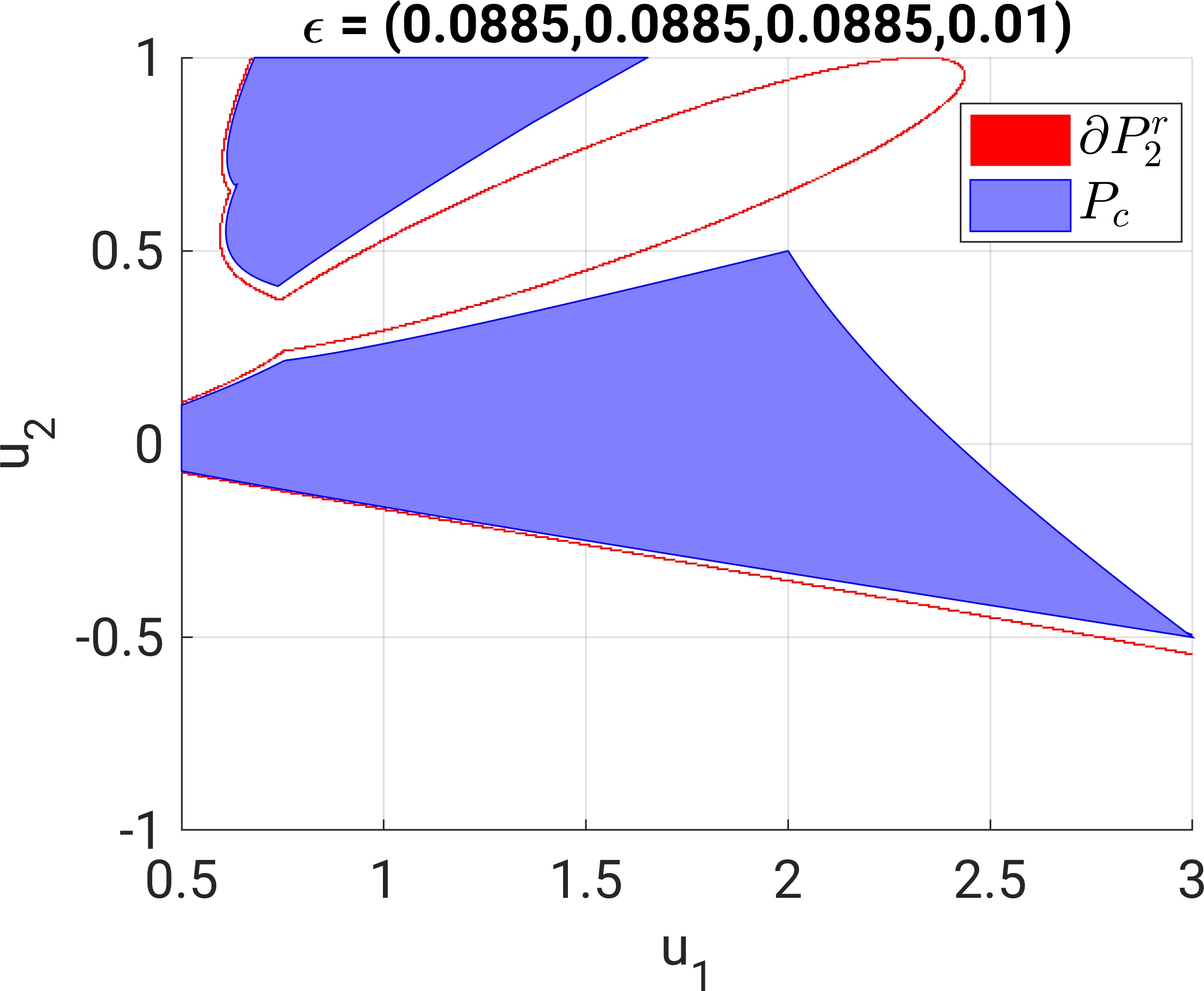}
			}
			\parbox[b]{0.32\textwidth}{

				\includegraphics[width=0.32\textwidth]{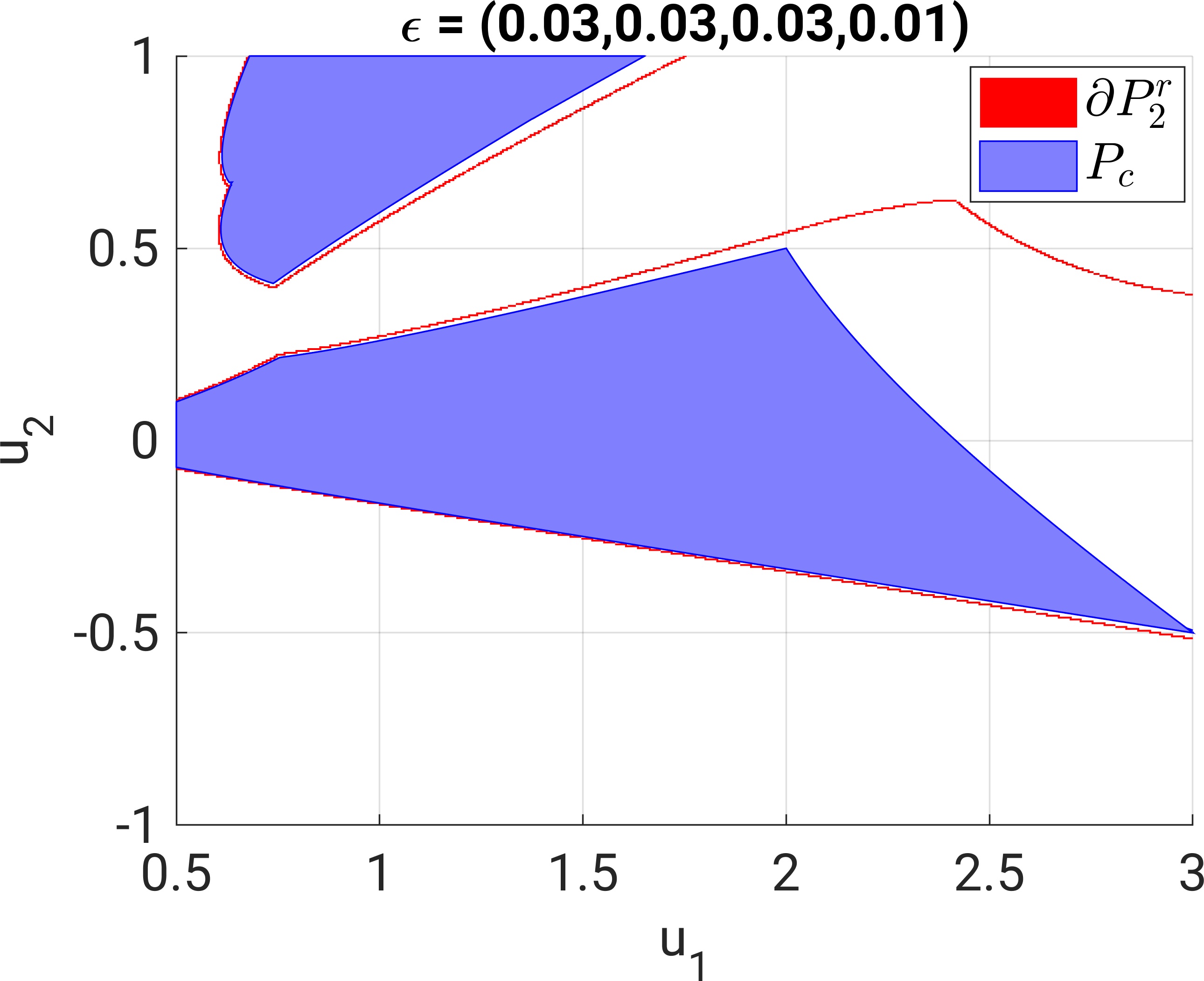}
			}
			\parbox[b]{0.32\textwidth}{

				\includegraphics[width=0.32\textwidth]{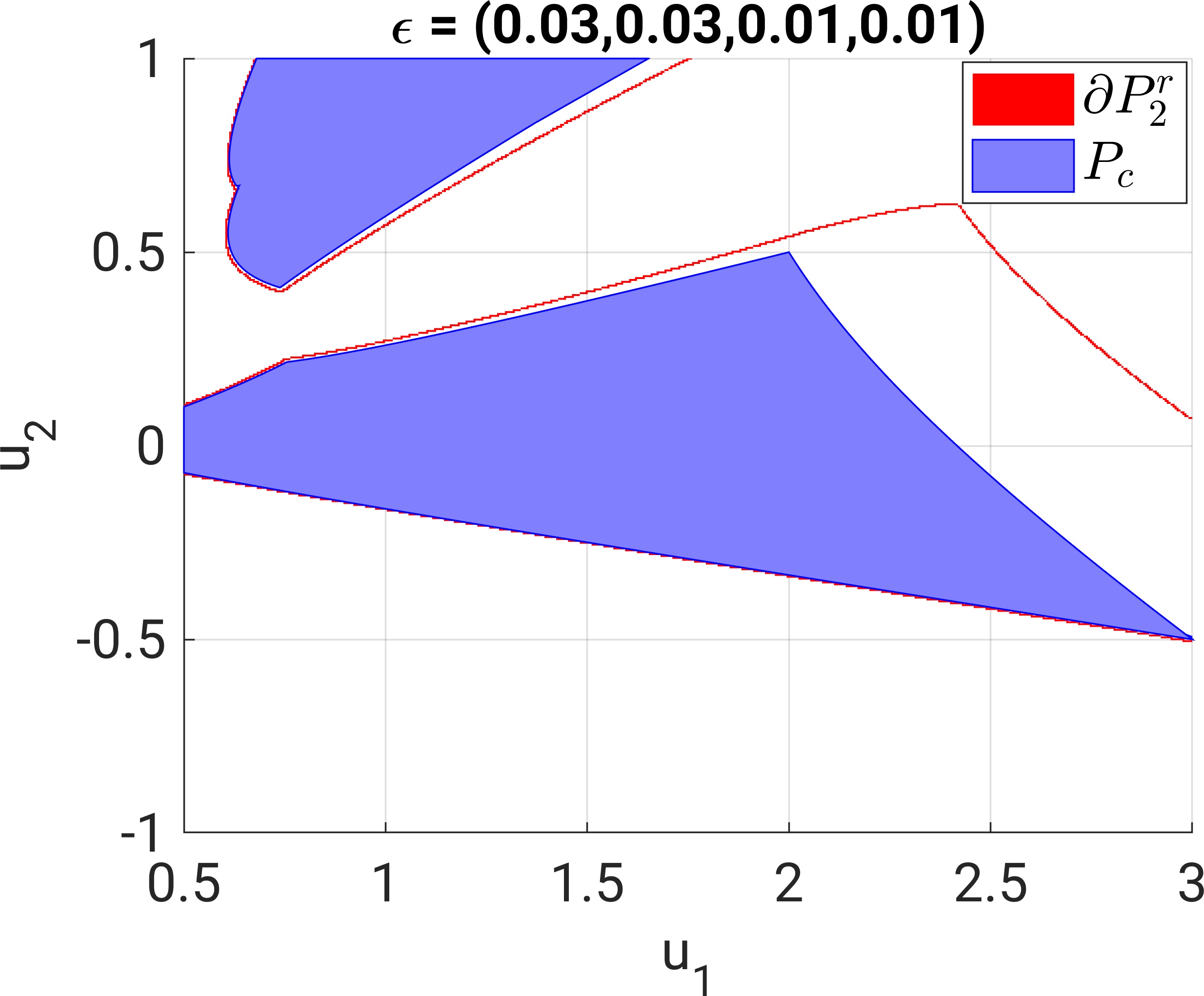}
			}
			\parbox[b]{0.32\textwidth}{

				\includegraphics[width=0.32\textwidth]{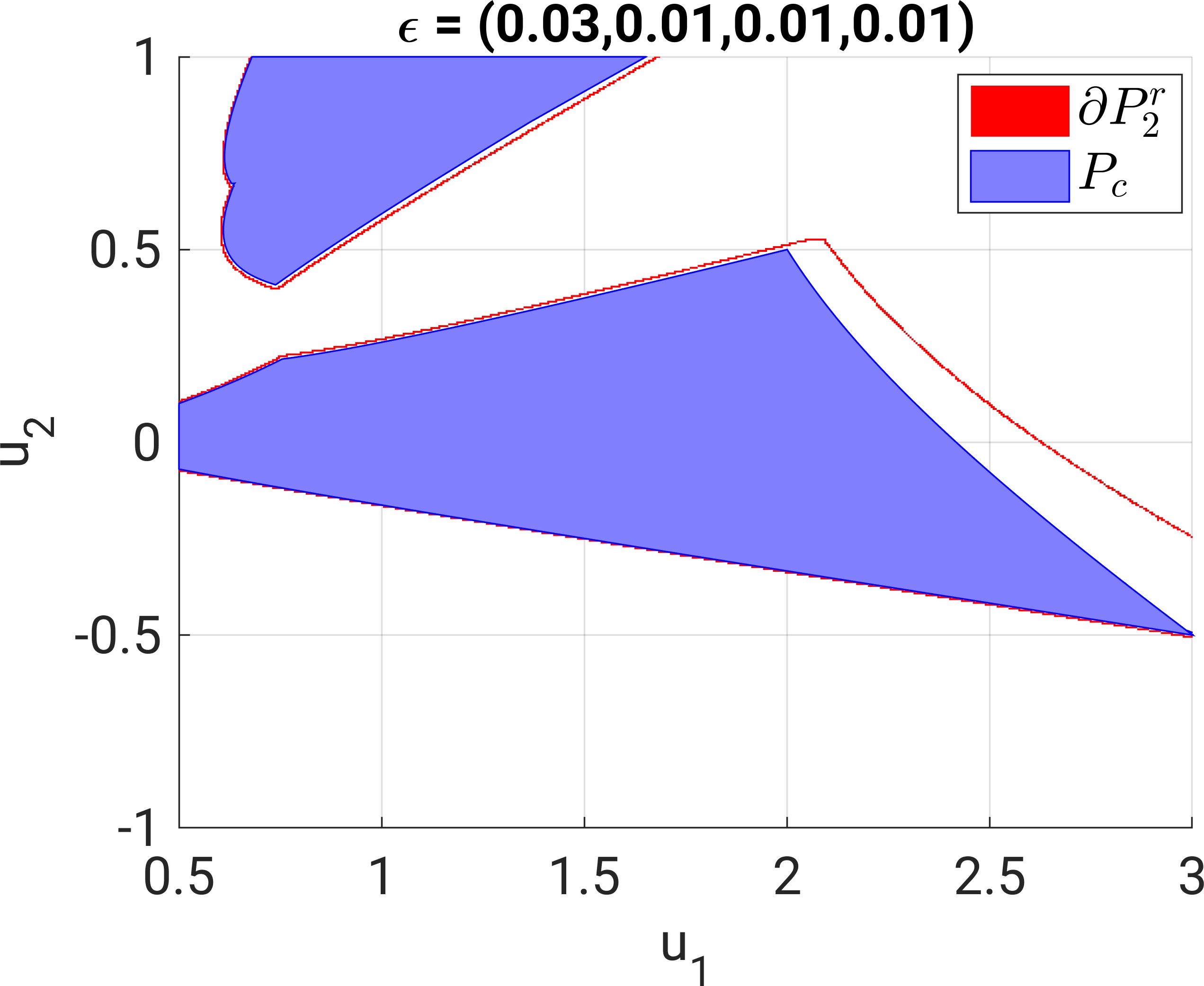}
			}
			\caption{Results of Strategy 2 for different values of $\epsilon$}
			\label{p5fig:Strategy2DifferentEpsilon}
		\end{figure}

	\section{Conclusion and outlook}		
	In this article, we present a way to efficiently solve multiobjective parameter optimization problems of elliptic PDEs by combining the reduced basis method from PDE-constrained optimization with the continuation approach from multiobjective optimization, which computes a box covering of the Pareto critical set. Using the RB method in this setting introduces an error in the objective functions and their gradients that has to be considered when solving the MPOP. To this end, we require that the reduced basis guarantees error bounds for the gradients of the objective functions. These error bounds are then incorporated into the KKT optimality conditions for MOPs to derive a tight superset $P_2^r$ of the actual Pareto (critical) set. This superset can be computed using a straightforward modification of the continuation method for MOPs (Strategy 1). Since $P_2^r$ has the same dimensions as the variable space of the MOP, we afterwards present a second method that only computes the boundary $\partial P_2^r$ of $P_2^r$ (Strategy 2). We do this by showing that $\partial P_2^r$ can be written as the level set of a differentiable mapping, which again enables the use of a continuation approach to compute it. For constructing the reduced basis, we use a greedy procedure which incorporates, and thus ensures, the error bounds for the gradients of the objective functions. \\
Our numerical tests show that the presented a-posteriori error estimator for the error in the gradients is not well-suited for the application in a greedy procedure due to its bad efficiency. Therefore, a strong greedy algorithm is used to build the reduced basis. Concerning the solution of the MPOP we investigate two aspects: First, the runtimes of our methods are compared. In our case, Strategy 1 is about 13 times and Strategy 2 about $63$ times faster than the exact solution of the MPOP (via the classical continuation method with FEM discretization). Second, the influence of the error bound for the gradients of the objective functions is investigated. As expected, a smaller error bound leads to a tighter covering of the Pareto critical set. Moreover, we observe that single components of the error bound strongly influence the tightness of the covering in areas, in which the corresponding components of the KKT-multipliers are large. Thus, by individually adapting the single components of the error bound, we can nicely control the tightness of the covering.
	
	For future work, there are some theoretical and practical aspects that should be investigated further:
	\begin{itemize}
		\item As mentioned in Remark \ref{p5rem:S2_subproblem}, in certain situations there can be difficulties when solving the problem \eqref{p5eq:bdryP2r_test}. In these situations, specialized methods that take these difficulties into account should be developed and used instead of standard methods for constrained optimization.
		\item If the number of objectives of the MPOP is larger than the number of variables, it may be possible to combine our approaches in this article with the hierarchical decomposition of the Pareto critical set presented in \cite{GPD2019}.
		\item The development of a more efficient a-posteriori error estimator for the error in the gradients of the objective functions would allow to use it in the greedy procedure. In that way, the expensive strong greedy procedure would be avoided in the offline phase. One way to do so might be the application of localized RB methods, see e.g. \cite{Ohlberger2017}.
		\item As explained in the globalization approach in Section \ref{p5sec:globalization}, we have to use multiple initial points to ensure that we find all connected components of $P_2^r$ (and faces of $\partial P_2^r$). Due to the local nature of the continuation method, this approach can potentially be parallelized, increasing the efficiency of our methods even more.
		\item If a decision maker is present with a certain preference, it may be worth to steer our continuation method in a direction that results from that preference instead of approximating the complete Pareto set. For the case with exact gradients, this was done in \cite{SCM2019}.
	\end{itemize}
		
	\bibliography{literature}   
	\bibliographystyle{abbrv}

	\subsection*{Acknowledgment}
This research was funded by the DFG Priority Programme 1962 ``Non-smooth and Complementarity-based Distributed Parameter Systems''.

	\appendix
	
	\section{Proof of Theorem \ref{p5thm:varphi_differentiable}}\label{p5app:ProofPhiDiff}
	
	To prove Theorem \ref{p5thm:varphi_differentiable}, we first have to investigate some of the properties of the optimization problem \eqref{p5eq:P2_problem}. This problem is quadratic with linear equality and inequality constraints. We will first investigate the uniqueness of the solution in the following lemma.

\begin{lemma} \label{p5lem:alpha_unique_1}
	Let $u \in \varphi^{-1}(0)$ and let $\alpha^1$ and $\alpha^2$ be two solutions of \eqref{p5eq:P2_problem} with $\alpha^1 \neq \alpha^2$. Then $\omega(\alpha) = 0$ for all $\alpha \in span( \{ \alpha^1, \alpha^2 \} )$ and 
	\begin{equation} \label{p5eq:unique_1}
	span( \{ \alpha^1, \alpha^2 \} ) \cap ker(DJ(u)^\top) \neq \emptyset.
	\end{equation}				
\end{lemma}
\begin{proof}
	For $c_1, c_2 \in \mathbb{R} \setminus \{ 0 \}$ we have
	\begin{align*}
		&\omega(c_1 \alpha^1 + c_2 \alpha^2) \\
		&= (c_1 \alpha^1 + c_2 \alpha^2)^\top (DJ^r(u) DJ^r(u)^\top - \epsilon \epsilon^\top) (c_1 \alpha^1 + c_2 \alpha^2) \\
		&= c_1^2 \omega(\alpha^1) + 2 ({c_1 \alpha^1}^\top DJ^r(u) DJ^r(u)^\top c_2 \alpha^2 - {c_1 \alpha^1}^\top \epsilon \epsilon^\top c_2 \alpha^2) + c_2^2 \omega(\alpha^2) \\
		&= 2 ({c_1 \alpha^1}^\top DJ^r(u) DJ^r(u)^\top c_2 \alpha^2 - {c_1 \alpha^1}^\top \epsilon \epsilon^\top c_2 \alpha^2) \\
		&= 2 c_1 c_2 ( (DJ^r(u)^\top \alpha^1)^\top (DJ^r(u)^\top \alpha^2) - (\epsilon^\top \alpha^1) (\epsilon^\top \alpha^2) ).
	\end{align*}
	From $\omega(\alpha^1) = \omega(\alpha^2) = 0$ it follows that $\epsilon^\top \alpha^1 = \| DJ^r(u)^\top \alpha^1 \|$ and $\epsilon^\top \alpha^2 = \| DJ^r(u)^\top \alpha^2 \|$. Let $\sphericalangle$ be the angle between $DJ^r(u)^\top \alpha^1$ and $DJ^r(u)^\top \alpha^2$. Then
	\begin{align} \label{p5eq:omega_add}
		&\omega(c_1 \alpha^1 + c_2 \alpha^2) \nonumber \\
		&= 2 c_1 c_2 (cos(\sphericalangle) \| DJ^r(u)^\top \alpha^1 \| \| DJ^r(u)^\top \alpha^2 \| - \| DJ^r(u)^\top \alpha^1 \| \| DJ^r(u)^\top \alpha^2 \|) \nonumber \\
		&= 2 c_1 c_2 (cos(\sphericalangle) - 1) \| DJ^r(u)^\top \alpha^1 \| \| DJ^r(u)^\top \alpha^2 \|. 
	\end{align}
	Assume $cos(\sphericalangle) \neq 1$ (i.e., $cos(\sphericalangle) - 1 < 0$), $\| DJ^r(u)^\top \alpha^1 \| \neq 0$ and $\| DJ^r(u)^\top \alpha^2 \| \neq 0$. If we choose $c_1 = t$ and $c_2 = 1 - t$ for $t \in (0,1)$, then $t \alpha^1 + (1-t) \alpha^2 \in \Delta_k$ and  $\omega(t \alpha^1 + (1-t) \alpha^2) < 0$, which contradicts $u \in \varphi^{-1}(0)$. If $\| DJ^r(u)^\top \alpha^1 \| = 0$ or $\| DJ^r(u)^\top \alpha^2 \| = 0$ then \eqref{p5eq:unique_1} holds for $\bar{\alpha} = \alpha^1$ or $\bar{\alpha} = \alpha^2$, respectively. If $cos(\sphericalangle) - 1 = 0$ then $DJ^r(u)^\top \alpha^1$ and $DJ^r(u)^\top \alpha^2$ are linearly dependent, so there are $\bar{c}_1$, $\bar{c}_2 \in \mathbb{R} \setminus \{ 0 \}$ such that $DJ^r(u)^\top \bar{\alpha} = 0$ for $\bar{\alpha} = \bar{c}_1 \alpha^1 + \bar{c}_2 \alpha^2$. In particular, in any case we must have $\omega(\alpha) = 0$ for all $\alpha \in span(\{ \alpha^1, \alpha^2  \})$.  
\end{proof}		

The previous lemma implies that for $k = 2$, the solution of \eqref{p5eq:P2_problem} for $u \in \varphi^{-1}(0)$ is non-unique iff $DJ^r(u) DJ^r(u)^\top - \epsilon \epsilon^\top = 0$. For $k > 2$, we can only have non-uniqueness if \eqref{p5eq:unique_1} holds. If we consider the dimensions of the spaces in \eqref{p5eq:unique_1}, we see that in the generic case, it can only hold if 
\begin{align*}
& dim(span( \{ \alpha^1, \alpha^2 \} ) \cap ker(DJ(u)^\top)) \geq 1 \\
\Leftrightarrow \ & 2 + k - rk(DJ(u)^\top) - k \geq 1 \\
\Leftrightarrow \ & rk(DJ(u)^\top) \leq 1,
\end{align*}
i.e., if all gradients of the objectives are linearly dependent in $u$. This motivates us to assume that in general, the solution of \eqref{p5eq:P2_problem} is unique for almost all $u \in \varphi^{-1}(0)$.

We will now investigate the differentiability of $\varphi$. Our strategy is to apply the implicit function theorem to the KKT conditions of \eqref{p5eq:P2_problem} to obtain a differentiable function $\phi$ that maps a point $u \in \mathbb{R}^n$ onto the solution of \eqref{p5eq:P2_problem} in $u$. This would imply the differentiability of $\varphi$ via concatenation with $\omega$. An obvious problem here is the fact that \eqref{p5eq:P2_problem} has inequality constraints which, when activated or deactivated under variation of $u$, lead to non-differentiabilities in $\phi$. Note that an inequality constraint being active means that one component of $\alpha$ is zero, i.e., one of the objective functions has no impact on the current problem. Thus, for our theoretical purposes, if there is an active inequality constraint in \eqref{p5eq:P2_problem} we will just ignore the corresponding objective function. This approach is strongly related to the hierarchical decomposition of the Pareto critical set (cf.~\cite{GPD2019}). 

For the reasons mentioned above, we will now consider the case where the solution of \eqref{p5eq:P2_problem} is strictly positive in each component. The following lemma shows a technical result that will be used in a later proof.
\begin{lemma} \label{p5lem:alpha_unique_2}
	Let $u \in \varphi^{-1}(0)$ and let $\bar{\alpha} \in \Delta_k$ be a solution of \eqref{p5eq:P2_problem} with $\alpha_i > 0$ $\forall i \in \{1,...,k\}$. Then $\bar{\alpha}$ is unique if and only if there is no $\beta \in \mathbb{R}^k \setminus \{ 0 \}$ with $\omega(\beta) = 0$ and $\sum_{i = 1}^k \beta_i = 0$.
\end{lemma}
\begin{proof}
	We will show that $\alpha$ is non-unique if and only if there is some $\beta \in \mathbb{R}^k$ with $\omega(\beta) = 0$ and $\sum_{i = 1}^k \beta_i = 0$. \\		
	$\Rightarrow$: Let $\tilde{\alpha}$ be another solution of \eqref{p5eq:P2_problem}. Then, as in the proof of Lemma \ref{p5lem:alpha_unique_1}, we must have $\omega(c_1 \bar{\alpha} + c_2 \tilde{\alpha}) = 0$ for all $c_1, c_2 \in \mathbb{R}$. This means we can choose $\beta = \bar{\alpha} - \tilde{\alpha}$. \\
	$\Leftarrow$: Let $\beta \in \mathbb{R}^k$ with $\omega(\beta) = 0$ and $\sum_{i = 1}^k \beta_i = 0$. Let $s > 0$ be small enough such that $\bar{\alpha} + s \beta \in \Delta_k$. Then, as in \eqref{p5eq:omega_add}, we have
	\begin{align*}
	\omega(\bar{\alpha} + s \beta) = 2 s (cos(\sphericalangle) - 1) \| DJ^r(u)^\top \bar{\alpha} \| \| DJ^r(u)^\top \beta \| \leq 0.
	\end{align*}
	Since by assumption $\varphi(u) = 0$ we must have $\omega(\bar{\alpha} + s \beta) = 0$, so $\bar{\alpha} + s \beta$ is another solution of \eqref{p5eq:P2_problem}.
\end{proof}			

To be able to use the KKT conditions of \eqref{p5eq:P2_problem} to obtain its solution, we have to make sure that these conditions are sufficient. Since \eqref{p5eq:P2_problem} is a quadratic problem, this means we have to show that the matrix in the objective $\omega$ is positive semidefinite.
\begin{lemma} \label{p5lem:P2_pos_semidef}
	Let $u \in \varphi^{-1}(0)$ and let $\bar{\alpha} \in \Delta_k$ be the unique solution of \eqref{p5eq:P2_problem} with $\bar{\alpha}_i > 0$ $\forall i \in \{1,...,k\}$. Then $\omega(\beta) \geq 0$ for all $\beta \in \mathbb{R}^k$. In particular, $DJ(u) DJ(u)^\top - \epsilon \epsilon^\top$ is positive semidefinite.
\end{lemma}		
\begin{proof}
	Assume there is some $\beta \in \mathbb{R}^k$ with $\omega(\beta) < 0$, i.e., $\epsilon^\top \beta > \| DJ^r(u)^\top \beta \|$. We distinguish between two cases: \\
	Case 1: $\sum_{i = 1}^k \beta_i = 0$: Similar to the proof of Lemma \ref{p5lem:alpha_unique_1} we get
	\begin{align*}
	\omega(\bar{\alpha} + s \beta) &< 2 s ( (DJ^r(u)^\top \bar{\alpha})^\top (DJ^r(u)^\top \beta) - (\epsilon^\top \bar{\alpha}) (\epsilon^\top \beta) ) \\
	&< 2 s (cos(\sphericalangle) - 1) \| DJ^r(u)^\top \bar{\alpha} \| \| DJ^r(u)^\top \beta \| \leq 0
	\end{align*}
	for all $s > 0$. In particular, since $\bar{\alpha}$ is positive, there is some $\bar{s} > 0$ such that $\bar{\alpha} + \bar{s} \beta \in \Delta_k$ with $\omega(\bar{\alpha} + \bar{s} \beta) < 0$, which is a contradiction. \\
	Case 2: $\sum_{i = 1}^k \beta_i \neq 0$. W.l.o.g. assume that $\sum_{i = 1}^k \beta_i = 1$. Consider
	\begin{equation*}
	\bar{\omega} : \mathbb{R} \rightarrow \mathbb{R}, \quad s \mapsto \omega(\bar{\alpha} + s (\beta - \bar{\alpha})).
	\end{equation*}
	Then $\bar{\omega}(0) = 0$ and $\bar{\omega}(1) < 0$. By assumption we must have $\bar{\omega}(s) > 0$ for all $s$ such that $\bar{\alpha} + s (\beta - \bar{\alpha}) \in \Delta_k$. By continuity of $\bar{\omega}$ there must be some $s^*$ with $\bar{\omega}(s^*) = 0$. Let $\bar{\beta} := \bar{\alpha} + s^* (\beta - \bar{\alpha})$. Using \eqref{p5eq:omega_add} we get
	\begin{align*}
	\omega(\bar{\alpha} + t s^* (\beta - \bar{\alpha})) &= \omega((1 - t) \bar{\alpha} + t \bar{ \beta}) \\
	&= 2 t (1-t) (cos(\sphericalangle) - 1) \| DJ^r(u)^\top \bar{\alpha} \| \| DJ^r(u)^\top \bar{\beta} \| \leq 0
	\end{align*} 
	for all $t \in (0,1)$, which is a contradiction.
\end{proof}			

The previous results now allow us to prove Theorem \ref{p5thm:varphi_differentiable}.
\paragraph{Theorem \ref{p5thm:varphi_differentiable}}
	\textit{
		Let $\bar{u} \in \varphi^{-1}(0)$ such that \eqref{p5eq:P2_problem} has a unique solution $\bar{\alpha} \in \Delta_k$ with $\bar{\alpha}_i > 0$ for all $i \in \{1,...,k\}$. Let \eqref{p5eq:P2_problem} be uniquely solvable in a neighborhood of $\bar{u}$. Then there is an open set $U \subseteq \mathbb{R}^n$ with $\bar{u} \in U$ such that $\varphi|_U$ is continuously differentiable.	
	}
\begin{proof}
	The KKT conditions for \eqref{p5eq:P2_problem} are
	\begin{align} \label{p5eq:P2_KKT}
	(DJ(u) DJ(u)^\top - \epsilon \epsilon^\top) \alpha - 
	\begin{pmatrix}
	\lambda + \mu_1 \\
	\vdots \\
	\lambda + \mu_k
	\end{pmatrix}
	&= 0, \nonumber \\
	\sum_{i = 1}^k \alpha_i  - 1 &= 0, \nonumber \\
	\alpha_i &\geq 0 \ \forall i \in \{1,...,k\}, \\
	\mu_i &\geq 0 \ \forall i \in \{1,...,k\}, \nonumber \\
	\mu_i \alpha_i &= 0 \ \forall i \in \{1,...,k\}. \nonumber
	\end{align}
	for $\lambda \in \mathbb{R}$ and $\mu \in \mathbb{R}^k$. By Lemma \ref{p5lem:P2_pos_semidef} these conditions are sufficient for optimality. By our assumption there is an open set $U'$ with $\bar{u} \in U'$ such that the solution of \eqref{p5eq:P2_problem} is unique and positive. Thus, on $U'$, \eqref{p5eq:P2_KKT} is equivalent to
	\begin{align*}
	(DJ(u) DJ(u)^\top - \epsilon \epsilon^\top) \alpha - 
	\begin{pmatrix}
	\lambda \\
	\vdots \\
	\lambda 
	\end{pmatrix}
	&= 0, \\
	\sum_{i = 1}^k \alpha_i  - 1 &= 0.
	\end{align*}
	for some $\lambda \in \mathbb{R}$. This system can be rewritten as $G(u,(\alpha,\lambda)) = 0$ for
	\begin{equation*}
	G : \mathbb{R}^n \times \mathbb{R}^{k+1} \rightarrow \mathbb{R}^{k+1}, \quad (u,(\alpha,\lambda)) \mapsto 
	\begin{pmatrix}
	(DJ(u) DJ(u)^\top - \epsilon \epsilon^\top) \alpha - (\lambda,...,\lambda)^\top \\
	\sum_{i = 1}^k \alpha_i  - 1
	\end{pmatrix}.
	\end{equation*}
	Derivating $G$ with respect to $(\alpha,\lambda)$ yields
	\begin{equation*}
	D_{(\alpha,\lambda)} G(u,(\alpha,\lambda)) = 
	\begin{pmatrix}
	(DJ(u) DJ(u)^\top - \epsilon \epsilon^\top) & (-1,...,-1)^\top \\
	(1,...,1) & 0
	\end{pmatrix}
	\in \mathbb{R}^{(k+1) \times (k+1)}.
	\end{equation*}
	Let $\bar{\lambda} \in \mathbb{R}$ such that $G(\bar{u},(\bar{\alpha},\bar{\lambda})) = 0$. (Note that uniqueness of $\bar{\alpha}$ implies uniqueness of $\bar{\lambda}$ here.) For $D_{(\alpha,\lambda)} G(\bar{u},(\bar{\alpha},\bar{\lambda}))$ to be singular, there would have to be some $v = (v^1,v^2) \in \mathbb{R}^{k+1}$ with	
	\begin{align*}
	0 = D_{(\alpha,\lambda)} G(\bar{u},(\bar{\alpha},\bar{\lambda})) v = 
	\begin{pmatrix}
	(DJ(\bar{u}) DJ(\bar{u})^\top - \epsilon \epsilon^\top) v^1 - (v^2,...,v^2)^\top \\
	\sum_{i = 1}^k v_i^1
	\end{pmatrix}
	\end{align*}			
	and thus
	\begin{align*}
	0 &= {v^1}^\top (DJ(\bar{u}) DJ(\bar{u})^\top - \epsilon \epsilon^\top) v^1 - {v^1}^\top (v^2,...,v^2)^\top \\
	&= {v^1}^\top (DJ(\bar{u}) DJ(\bar{u})^\top - \epsilon \epsilon^\top) v^1 - v_2 \sum_{i = 1}^k v_i^1 \\
	&= w(v^1).
	\end{align*}
	By Lemma \ref{p5lem:alpha_unique_2}, this is a contradiction to the assumption that $\bar{\alpha}$ is a unique solution of \eqref{p5eq:P2_problem}. So $D_{(\alpha,\lambda)} G(\bar{u},(\bar{\alpha},\bar{\lambda}))$ has to be regular. This means we can apply the implicit function theorem to obtain open sets $U \subseteq U' \subseteq \mathbb{R}^n$, $V \subseteq \mathbb{R}^{k+1}$ with $\bar{u} \in U$, $(\bar{\alpha},\bar{\lambda}) \in V$ and a continuously differentiable function $\phi = (\phi_\alpha, \phi_\lambda): U \rightarrow V$ with 
	\begin{equation*}
	G(u,(\alpha,\lambda)) = 0 \ \Leftrightarrow \ (\alpha,\lambda) = \phi(u) \quad \forall u \in U,(\alpha,\lambda) \in V. 
	\end{equation*}
	In particular,
	\begin{equation} \label{p5eq:varphi_phi}
	\varphi|_U(u) = \min_{\alpha \in \Delta_k} \left( \| DJ(u)^\top \alpha \|^2 - (\alpha^\top \epsilon)^2 \right) = \| DJ(u)^\top \phi_\alpha(u) \|^2 - (\phi_\alpha(u)^\top \epsilon)^2,
	\end{equation}
	so $\varphi|_U$ is continuously differentiable.
\end{proof}		
	
\begin{remark}
	From the proof of Theorem \ref{p5thm:varphi_differentiable} we can even derive an explicit formula for the derivative of $\varphi|_U$ in $\bar{u}$: First of all, the derivative of the implicit function $\phi$ is given by
	\begin{align*}
		&D\phi(\bar{u}) \\
		&= - G_{(\alpha,\lambda)}(\bar{u},(\bar{\alpha},\bar{\lambda}))^{-1} G_u(\bar{u},(\bar{\alpha},\bar{\lambda})) \\
		&= -
		\begin{pmatrix}
		DJ(\bar{u}) DJ(\bar{u})^\top - \epsilon \epsilon^\top & -1_{k \times 1} \\
		1_{1 \times k} & 0
		\end{pmatrix}^{-1} \cdot \\
		& \hspace{60pt}  \left(		
		\begin{pmatrix}
		\bar{\alpha}^\top DJ(\bar{u}) \nabla^2 J_1(\bar{u}) \\
		\vdots \\
		\bar{\alpha}^\top DJ(\bar{u}) \nabla^2 J_k(\bar{u}) \\
		0_{1 \times n}
		\end{pmatrix}
		+
		\begin{pmatrix}
		DJ(\bar{u}) \sum_{i = 1}^k \bar{\alpha}_i \nabla^2 J_i(\bar{u}) \\
		0_{1 \times n}
		\end{pmatrix}
		\right).
	\end{align*}
	By applying the chain rule to \eqref{p5eq:varphi_phi} we obtain
	\begin{align} \label{p5eq:D_phi_expl}
	&D\varphi|_U(\bar{u}) \nonumber \\
	&= 2 (DJ(\bar{u})^\top \bar{\alpha})^\top \sum_{i = 1}^k \bar{\alpha}_i \nabla^2 J_i(\bar{u}) + \left( 2 (DJ(\bar{u})^\top \bar{\alpha})^\top DJ(\bar{u})^\top - 2 (\bar{\alpha}^\top \epsilon) \epsilon^\top \right) D\phi_\alpha(\bar{u}) \nonumber \\ 
	&= 2 (DJ(\bar{u})^\top \bar{\alpha})^\top \sum_{i = 1}^k \bar{\alpha}_i \nabla^2 J_i(\bar{u}) + 2 \bar{\alpha}^\top \left( DJ(\bar{u}) DJ(\bar{u})^\top -  \epsilon \epsilon^\top \right) D\phi_\alpha(\bar{u}) \nonumber \\
	&= 2 (DJ(\bar{u})^\top \bar{\alpha})^\top \sum_{i = 1}^k \bar{\alpha}_i \nabla^2 J_i(\bar{u}) + 2 (\bar{\lambda},...,\bar{\lambda}) D\phi_\alpha(\bar{u}).
	\end{align} \hfill$\blacksquare$
\end{remark}

\end{document}